\documentclass[a4paper,oneside,11pt]{article}
\usepackage{amsmath}
\usepackage{amssymb}
\usepackage[ngerman, english]{babel}
\usepackage[latin1]{inputenc}
\usepackage{MnSymbol}
\usepackage{bbm}
\usepackage{nicefrac}
\usepackage{epsfig}
\usepackage{subfigure}
\usepackage{tabularx}
\usepackage{amsthm}

\newcommand{\R}{\mathbbm{R}}
\newcommand{\N}{\mathbbm{N}}
\newcommand{\C}{\mathcal{C}}
\def\b#1{\boldsymbol{#1}}
\def\dx{\,\mathrm dx}

\def\ds{\,\mathrm ds}
\def\div{\,\mathrm{div}\,}

\def\der{\mathrm{D}}
\def\Epsilon{\mathcal E}

\def\epsilon{\varepsilon}
\newcommand{\bfphi}{\varphi}

\newcommand{\spr}[1]{\left\langle #1 \right\rangle} %Skalarprodukt
\newcommand{\bfv}{v}

\newtheorem{theorem}{Theorem}
\newtheorem{lemma}[theorem]{Lemma}
\newtheorem{bsp}[theorem]{Example}
\newtheorem{corollary}[theorem]{Corollary}
\newtheorem{remark}[theorem]{Remark}
\newtheorem{algo}{Algorithm}
\theoremstyle{definition}
\newtheorem{definition}[theorem]{Definition}

\newcounter{AssCount}
\renewcommand{\theAssCount}{\textbf{(A{\arabic{AssCount}})}}
\newcounter{AssListCount}

\usepackage{algorithmic}
\usepackage[Algorithm, section]{algorithm}

\begin{document}

\title{Sharp interface limit for a phase field model in structural optimization}

%\author{Claudia Hecht\footnotemark[1]}
%\author{Harald Garcke, Claudia Hecht\footnote{Fakult\"at f\"ur Mathematik, Universit\"at Regensburg, 93040 Regensburg, Germany,
%e-mail: Claudia.Hecht@mathematik.uni-regensburg.de}}

\author{Luise Blank\footnotemark[1]\and Harald Garcke\footnotemark[1]\and Claudia Hecht\footnotemark[1]\and Christoph Rupprecht\footnotemark[1]}

\date{}

\maketitle

\renewcommand{\thefootnote}{\fnsymbol{footnote}}
%
%\footnotetext[1]{Fakult\"at f\"ur Mathematik, Universit\"at Regensburg, 93040 Regensburg, Germany,
%{\tt Claudia.Hecht@mathematik.uni-regensburg.de}.}
\footnotetext[1]{Fakult\"at f\"ur Mathematik, Universit\"at Regensburg, 93040 Regensburg, Germany
({\tt \{Luise.Blank, Harald.Garcke, Claudia.Hecht, Christoph.Rupprecht\}@mathematik.uni-regensburg.de}).}
\renewcommand{\thefootnote}{\arabic{footnote}}

\begin{abstract}
\noindent We formulate a general shape and topology optimization problem in structural optimization by using a phase field approach. This problem is considered in view of well-posedness and we derive optimality conditions. We relate the diffuse interface problem to a perimeter penalized sharp interface shape optimization problem in the sense of $\Gamma$-convergence of the reduced objective functional. Additionally, convergence of the equations of the first variation can be shown. The limit equations can also be derived directly from the problem in the sharp interface setting. Numerical computations demonstrate that the approach can be applied for complex structural optimization problems.
\end{abstract}



\pagestyle{myheadings}
\markboth{L. BLANK, H. GARCKE, C. HECHT, C. RUPPRECHT}{PHASE FIELD PROBLEMS IN STRUCTURAL OPTIMIZATION}

\section{Introduction}
In structural optimization one tries to find an optimal material configuration of two different elastic materials in some fixed container, where optimal means that a certain objective functional depending on the behaviour of the elastic materials is minimized. The control here is represented by the material distribution. Applications of shape and topology optimization reach from crashworthiness of transport vehicles and tunnel design to biomechanical applications such as bone remodelling. Structural optimization has turned out to be helpful in solving automative design problems in order to maximize the stiffness of vehicles for instance or to reduce the stresses to improve durability, see for instance \cite{bendsoe2003topology}.\\
One of the first approaches of finding the optimal material distribution in presence of two materials can be found in \cite{thomson}. However the problem of finding optimal structures in mechanical engineering dates at least back to the beginning of the 20th century when Michell \cite{michell} considered optimal truss layouts. It has turned out, that generally those problems are not well-posed, because oscillations occur on a very fine scale, see for example \cite{haberjog}, and hence several ideas have been developed to overcome this issue. One important contribution is certainly the idea of using a perimeter penalization in optimal shape design and considering this problem in the framework of Caccioppoli sets, see \cite{ambrosioButtazzo}, in order to prevent the above-mentioned oscillations. Additionally, it turns out that it is difficult to control the state variables if they are only given on varying domains of definitions. And so a so-called ersatz material approach has been introduced, see for instance \cite{allaire_jouve, bourdin_chambolle}. Here, one replaces the void regions by a fictitious material which may have a very low stiffness. We also remark that there appear many problems of practical relevance where one wants to fill a given domain with two different materials (and not one material and void) such that after an applied load an objective functional is minimized. Having this in mind it is the main goal of this paper to analyze and numerically solve problems with two materials, whether fictious or not.\\
We start by stating a perimeter penalized shape optimization problem with a general objective functional in Section~\ref{s:ElastSharpProblemForm}. This is in a simplified form given as
\begin{equation}\label{e:FirstIntroProblem}\begin{split}&\min_{(\varphi,\b u)}J_0(\varphi,\b u):=\int_\Omega h_\Omega(x,\b u)\dx+\int_{\Gamma_g}h_\Gamma(s,\b u)\ds+\gamma c_0 P_\Omega(\{\varphi=1\})\\
&\text{subject to }\int_\Omega \C(\varphi)\left(\Epsilon\left(\b u\right)-\overline\Epsilon\left(\varphi\right)\right):\Epsilon(\b v)\dx=\int_\Omega \b f\cdot\b v\dx+\int_{\Gamma_g}\b g\cdot\b v\ds\quad\forall\b v.\end{split}\end{equation}
After showing well-posedness we derive necessary optimality conditions by geometric variations without any additional regularity assumption on the minimizing set other than being a Caccioppoli set. This seems to be new as classical shape derivatives always assume at least an open Lipschitz domain as minimizer, see \cite{allairemulti, Allaire2004363, allaire2010damage, sturm}, and they do not treat a general objective functional. We also show that the obtained conditions are consistent with existing results obtained with shape derivatives if the minimizing shape inherits a certain regularity.\\
Then we approximate this problem by using a phase field approach where the free boundary is replaced by a diffuse interface with small thickness related to a parameter $\epsilon>0$. Hence, as in \cite{bourdin_chambolle}, the perimeter functional is replaced by the Ginzburg-Landau energy and the optimization problem $\eqref{e:FirstIntroProblem}$ reads as
\begin{equation}\label{e:FirstIntroProblemPhase}\begin{split}&\min_{(\varphi,\b u)}J_\epsilon(\varphi,\b u):=\int_\Omega h_\Omega(x,\b u)\dx+\int_{\Gamma_g}h_\Gamma(s,\b u)\ds+\gamma\int_\Omega\frac\epsilon2\left|\nabla\varphi\right|^2+\frac1\epsilon\psi\left(\varphi\right)\dx\\
&\text{subject to }\int_\Omega \C(\varphi)\left(\Epsilon\left(\b u\right)-\overline\Epsilon\left(\varphi\right)\right):\Epsilon(\b v)\dx=\int_\Omega \b f\cdot\b v\dx+\int_{\Gamma_g}\b g\cdot\b v\ds\quad\forall\b v.\end{split}\end{equation}

 After discussing well-posedness and necessary optimality conditions for the phase field problem we consider the sharp interface limit. To be precise, we show $\Gamma$-convergence of the reduced objective functional as the interfacial width, i.e. $\epsilon$, tends to zero.  Moreover, we show that the equations of the optimality systems converge. We hereby generalize findings from literature where this result has already been indicated in \cite{relatingphasefield} by formal asymptotics for certain objective functionals.\\
The paper is structured as follows: In Section~\ref{s:DiscProbleResults} we introduce the exact problem formulations, discuss well-posedness, optimality conditions and the sharp interface limit. The derivation of the optimality conditions can be found in Section~\ref{s:ProofOptcond} and some proofs of the sharp interface convergence results are collected in Section~\ref{s:ProofConvRes}. The numerical approach and results are given in Section \ref{s:Numerics}.

\section{Discussion of the problems and convergence results}\label{s:DiscProbleResults}

\subsection{Notation and assumptions}\label{s:NotAss}
Before formulating the shape optimization problems we give a brief introduction into the most important quantities and equations in linearized elasticity and fix some notation. We refer the reader to \cite{braess, ciarlet, eckgarcke} and references therein for details. We first assume to have in the holdall container $\Omega$ two open subsets $\Omega_1$ and $\Omega_2$ which are separated by a hypersurface $\Gamma=\partial\Omega_1\cap\partial\Omega_2$. The two subsets should correspond to two different elastic materials whose displacement fields are described by one variable $\b u:\Omega\to\R^d$. To be precise, $\b u|_{\Omega_i}$ corresponds to the displacement field of the $i$-th material where $i\in\{1,2\}$. We divide the boundary of $\Omega$ into two parts, one Dirichlet part where we can prescribe the displacement field, and a Neumann part where the applied boundary forces are acting. 

\begin{list}{\theAssCount}{\usecounter{AssCount}}\setcounter{AssCount}{\value{AssListCount}
}
\item \label{a:ElastOmega} $\Omega\subset\R^d$ is a bounded Lipschitz domain with outer unit normal $\b n$ and $d\in\{2,3\}$. Moreover, assume $\partial\Omega=\Gamma_D\cup\Gamma_g$ with $\mathcal H^{d-1}\left(\Gamma_D\right)>0$ and $\Gamma_D\cap\Gamma_g=\emptyset$.
\setcounter{AssListCount}{\value{AssCount}}
\end{list}
We remark that we denote $\R^d$-valued functions and spaces consisting of $\R^d$-valued functions in boldface.\\
For elastic materials the following equilibrium constraints hold in $\Omega_i$, $i\in\{1,2\}$:
\begin{subequations}\label{e:ElastEquilibraiumCaucyh}\begin{align}
	-\nabla\cdot\left(\der_2 W_i\left(x,\Epsilon\left(\b u\right)\right)\right)&=\b f &&\text{in }\Omega_i,\\
	\der_2W_i\left(x,\Epsilon\left(\b u\right)\right)\cdot\b n&=\b g &&\text{on }\Gamma_g\cap\partial\Omega_i,\\
	\b u&=\b u_D &&\text{on }\Gamma_D\cap\partial\Omega_i,
\end{align}\end{subequations}
where:
\begin{list}{\theAssCount}{\usecounter{AssCount}}\setcounter{AssCount}{\value{AssListCount}
}
\item \label{a:ElastForces}
$\b g\in\b L^2\left(\Gamma_g\right)$ is the given applied surface load, $\b f\in\ L^2\left(\Omega\right)$ the given applied surface load and for simplicity we assume for the following considerations $\b u_D\equiv\b 0$.
\setcounter{AssListCount}{\value{AssCount}}
\end{list}
On the interface $\Gamma:=\partial\Omega_1\cap\partial\Omega_2$ the boundary conditions are given by certain transmission properties, which follow from $\eqref{e:ElastPhaseFirstState}$. Moreover, $\Epsilon\left(\b u\right):=\frac12\left(\nabla\b u+\nabla\b u^T\right)$ is the so-called linearized strain, whereon the linear theory is based and $W_i:\Omega\times\R^{d\times d}\to\R$ denotes the elastic free energy density of the $i$-th material. We use $W_i\left(x,\Epsilon\right):=\frac12\left(\Epsilon-\overline\Epsilon_i\right):\C_i\left(\Epsilon-\overline\Epsilon_i\right)$ for $\Epsilon\in\R^{d\times d}$, which is in our case independent of $x\in\R^d$. Here, $\C_i:\R^{d\times d}\to\R^{d\times d}$ is the elasticity tensor reflecting the material properties for material $i=1$ and $i=2$, respectively. Further, $\overline\Epsilon_i\in\R^{d\times d}$ is the eigenstrain which is given as the value of the strain when the $i$-th material is unstressed. By $\der_2W_i$ we denote the derivative with respect to the second component.\\

As already mentioned above, we have two different elastic materials inside the domain $\Omega$. The design variable is a measurable function $\varphi:\Omega\to\R$, where $\{x\in\Omega\mid\varphi(x)=1\}=\Omega_1$ describes the region where the first material is present up to a set of measure zero, and $\left\{x\in\Omega\mid\varphi(x)=-1\right\}=\Omega_2$ the region which is filled with the second material. In the sharp interface setting, $\varphi$ will only take values in $\{\pm1\}$ and thus $\Omega=\Omega_1\cup\Omega_2\cup\Gamma$ with $\Gamma=\partial\Omega_1\cap\partial\Omega_2$ being the separating hypersurface. In contrast, the phase field approximation uses a design function $\varphi$ having values in $[-1,1]$. Then, $\Omega=\Omega_1\cup\Omega_2\cup I$ where $I=\{-1<\varphi<1\}$ is the diffuse interface of small thickness approximating the hypersurface $\Gamma$. Using $\varphi$ to describe the sharp as well as the diffuse interface model we describe the elasticity tensor and the eigenstrain as functions of the design variable $\varphi$ which interpolate between two different values for the two different materials. We introduce the following assumptions on the elasticity tensor $\C$ and we use the following assumptions: %For simplify notation, we assume in the following that $\overline\Epsilon x$ is a linear function, thus $\frac12\overline\Epsilon_1+\frac12\overline\Epsilon_2=0$. As this is only an affine transformation, it does not change anything in the following analysis. And so we can identify $\overline\Epsilon$ with a matrix in $\R^{d\times d}$, which is then assumed to be symmetric. 

\setcounter{AssListCount}{\value{AssCount}}
\begin{list}{\theAssCount}{\usecounter{AssCount}}\setcounter{AssCount}{\value{AssListCount}
}
\item \label{a:ElastTen} Let $\C(\varphi)=\left(\C_{ijkl}(\varphi)\right)_{i,j,k,l=1}^d$ be such that $\C_{ijkl}\in C^{1,1}\left(\left[-1,1\right]\right)$ fulfills pointwise the following symmetry properties
 $\C_{ijkl}(\varphi)=\C_{jikl}(\varphi)=\C_{klij}(\varphi)$ for all $\varphi\in[-1,1],$ $i,j,k,l\in\{1,\ldots,d\}$. Moreover, we assume that there exist constants $C_C,c_C>0$ such that
\begin{align}\label{e:UniforEstimateElast}
\left|\C(\varphi) A:B\right|\leq C_C\left|A\right|\left|B\right|,\quad \C(\varphi) A:A\geq c_C\left|A\right|^2
\end{align}
holds for all symmetric matrices $A,B\in\R^{d\times d}$ and $\varphi\in\left[-1,1\right]$.
\setcounter{AssListCount}{\value{AssCount}}
\end{list}

\setcounter{AssListCount}{\value{AssCount}}
\begin{list}{\theAssCount}{\usecounter{AssCount}}\setcounter{AssCount}{\value{AssListCount}
}
\item \label{a:Eigenstrain} Let the eigenstrain $\overline\Epsilon\in C^{1,1}\left(\left[-1,1\right],\R^{d\times d}\right)$ be a function with symmetric values, i.e. $\overline\Epsilon(\varphi)^T=\overline\Epsilon(\varphi)$ for all $\varphi\in[-1,1]$.
\setcounter{AssListCount}{\value{AssCount}}
\end{list}

\begin{remark}
\begin{enumerate}
\item The estimates $\eqref{e:UniforEstimateElast}$ imply that the elasticity tensor interpolates between two finite positive definite tensors, thus in particular no ``void'', i.e. regions without material, are allowed in this formulation. Anyhow, the possibility of modelling ``void'' is given by using the so-called ersatz material approach, where a very soft material approximates the non-presence of material, cf. \cite{relatingphasefield, bourdin_chambolle}. Moreover, the elasticity tensor is not depending on the phase field variable $\epsilon>0$ introduced later on. Hence, an ersatz material approach depending on the phase field parameter $\epsilon>0$ as it is employed in \cite{relatingphasefield} cannot be used in this setting.
\item Following Vegard's law, a commonly used assumption is that the eigenstrain interpolates linearly between the two values corresponding to the two materials. Then, Assumption \ref{a:Eigenstrain} is fulfilled.
\end{enumerate}
\end{remark}

Using these assumptions, a weak formulation of the state equations on the whole of $\Omega$ can be derived, if the design variable is $\varphi\in L^1(\Omega)$ with $\left|\varphi\right|\leq 1$ a.e. in $\Omega$:\\
Find $\b u\in\b H^1_D\left(\Omega\right):=\left\{\b u\in\b H^1(\Omega)\mid\b u|_{\Gamma_D}=\b 0\right\}$ such that
\begin{align}\label{e:ElastPhaseFirstState}
\int_\Omega\C\left(\varphi\right)\left(\Epsilon\left(\b u\right)-\overline\Epsilon\left(\varphi\right)\right):\Epsilon\left(\b v\right)\dx=\int_\Omega\b f\cdot\b v\dx+\int_{\Gamma_g}\b g\cdot\b v\ds\quad\forall\b v\in\b H^1_D(\Omega).
\end{align}
In any subregion $\{\varphi=\pm1\}$ this yields exactly the weak formulation of $\eqref{e:ElastEquilibraiumCaucyh}$. The state equation is in both the phase field and the sharp interface formulation given by $\eqref{e:ElastPhaseFirstState}$, cf. Sections~\ref{s:ElastSharpProblemForm} and~\ref{s:ElastPhase}. Hence we directly state here the solvability result concerning this equation:

\begin{lemma}\label{l:SolOp}
	For every $\varphi\in L^1(\Omega)$ with $\left|\varphi\right|\leq1$ a.e. in $\Omega$ there exists a unique $\b u\in\b H^1_D(\Omega)$ such that $\eqref{e:ElastPhaseFirstState}$ is fulfilled. Moreover, the solution $\b u$ fulfills 
	\begin{align}\label{e:ElastPHaseSTateEquaAprioriestimate}\left\|\b u\right\|_{\b H^1(\Omega)}\leq C\left(\Omega,\C,\overline\Epsilon\right)\left(\left\|\b f\right\|_{\b L^2(\Omega)}+\left\|\b g\right\|_{\b L^2(\Gamma_g)}+1\right).\end{align}
	This defines a solution operator $\b S:\{\varphi\in L^1(\Omega)\mid|\varphi|\leq1\text{ a.e. in }\Omega\}\to\b H^1_D(\Omega)$.
\end{lemma}
	{\em Idea of the proof}.	
	By making use of Korn's inequality, this result is a direct consequence of Lax-Milgram's theorem, cf. \cite[Lemma 24.1]{hecht}.\qquad$\square$

For our shape and topology optimization problem, the goal is to minimize
\begin{align}\label{e:ObjFctlOneReg}\mathbb H(\b u):=\int_\Omega h_\Omega\left(x,\b u\right)\dx+\int_{\Gamma_g}h_\Gamma\left(s,\b u\right)\ds\end{align}
where $h_\Omega, h_\Gamma$ fulfill

\begin{list}{\theAssCount}{\usecounter{AssCount}}\setcounter{AssCount}{\value{AssListCount}
}
 	\item\label{a:ElastObjectiveFctl}
  $h_\Omega:\Omega\times\R^d\to\R$ and $h_\Gamma:\Gamma_g\times\R^d\to\R$ are Carath\'eodory functions, i.e.
 	\begin{enumerate}
 	\item $h_\Omega(\cdot,\b v):\Omega\to\R$ and $h_\Gamma(\cdot,\b v):\Gamma_g\to\R$ are measurable for each $\b v\in\R^d$, and
 	\item $h_\Omega(x,\cdot), h_\Gamma(s,\cdot):\R^d\to\R$ are continuous for almost every $x\in\Omega$ and $s\in\Gamma_g$, respectively.
 	\end{enumerate}
 	Moreover, assume that there exist functions $a_1\in L^1(\Omega)$, $a_2\in L^1(\Gamma_g)$ and $b_1\in L^\infty(\Omega)$, $b_2\in L^\infty(\Gamma_g)$ such that it holds
	\begin{align}\label{e:ElastPhaseObjFctGrowF}|h_\Omega(x,\b v)|\leq a_1(x)+b_1(x)|\b v|^2\quad\forall\b v\in\R^d,\text{ a.e. }x\in\Omega,\end{align}
	and
	\begin{align}\label{e:ElastPhaseObjFctGrowG}|h_\Gamma(s,\b v)|\leq a_2(s)+b_2(s)|\b v|^2\quad\forall \b v\in\R^d,\text{ a.e. }s\in\Gamma_g.\end{align}
	Additionally, we assume that the set
	\begin{equation*}\begin{split}\left\{\int_\Omega h_\Omega\left(x,\b S(\varphi)(x)\right)\dx+\int_{\Gamma_g}h_\Gamma\left(s,\b S(\varphi)(s)\right)\ds\mid\varphi\in L^1(\Omega), |\varphi|\leq 1\text{ a.e. in }\Omega\right\}\end{split}\end{equation*}
is bounded from below.
\setcounter{AssListCount}{\value{AssCount}}
\end{list}

\begin{remark}\label{r:ContObjFctlElast}
Due to \cite{showalter}, the Nemytskii operators
$$L^2(\Omega)^d\ni\b v\mapsto h_\Omega\left(\cdot,\b v\right)\in L^1(\Omega),\quad L^2(\Gamma_g)^d\ni\b v\mapsto h_\Gamma\left(\cdot,\b v\right)\in L^1(\Gamma_g)$$
are well-defined if and only if $\eqref{e:ElastPhaseObjFctGrowF}$ and $\eqref{e:ElastPhaseObjFctGrowG}$ are fulfilled and in this case the operators are continuous.
\end{remark}

Assumptions \ref{a:ElastOmega}-\ref{a:ElastObjectiveFctl} are the basic assumptions for the following considerations. To derive also first order optimality conditions we have to impose at some points additionally the following differentiability assumption. 

\begin{list}{\theAssCount}{\usecounter{AssCount}}\setcounter{AssCount}{\value{AssListCount}
}

 	\item\label{a:ElastDiffAss} For every fixed $\b v\in\R^d$ we have $h_\Omega(\cdot,\b v)\in W^{1,1}(\Omega)$ and $h_\Gamma\left(\cdot,\b v\right)\in W^{1,1}(\Gamma_g)$. Let the partial derivatives $\der_2h_\Omega\left(x,\cdot\right), \der_2h_\Gamma\left(s,\cdot\right)$ exist for almost every $x\in\Omega$ and $s\in\Gamma_g$, respectively. Moreover there exist $\hat a_1\in L^2(\Omega)$, $\hat a_2\in L^2(\Gamma_g)$ and $\hat b_1\in L^\infty(\Omega)$, $\hat b_2\in L^\infty(\Gamma_g)$ such that
 		\begin{align}\label{e:ElastPhaseObjFctGrowFDer}|\der_{2}h_\Omega\left(x,\b v\right)|\leq \hat a_1(x)+\hat b_1(x)|\b v|\quad\forall\b v\in\R^d, \text{ a.e. }x\in\Omega\end{align}
 		and
 		\begin{align}|\der_{2}h_\Gamma\left(s,\b v\right)|\leq \hat a_2(s)+\hat b_2(s)|\b v|\quad\forall\b v\in\R^d,\text{ a.e. }s\in\Gamma_g.\end{align}
		\setcounter{AssListCount}{\value{AssCount}}
\end{list}

%\end{list}

\begin{remark}\label{r:ElastObjFctFrechet}
	Under the Assumption \ref{a:ElastDiffAss} the operators 	
	$$F:\b L^2\left(\Omega\right)\ni\b u\mapsto\int_\Omega h_\Omega\left(x,\b u(x)\right)\dx,\quad G:\b L^2(\Gamma_g)\ni\b u\mapsto\int_{\Gamma_g} h_\Gamma\left(s,\b u(s)\right)\ds$$
 	are continuously Fr\'{e}chet differentiable and that the directional derivatives are given by $\der F\left(\b u\right)\left(\b v\right)=\int_\Omega\der_2 h_\Omega\left(x,\b u\right)\b v\dx$, $\der G\left(\b u\right)\left(\b v\right)=\int_{\Gamma_g}\der_2h_\Gamma\left(s,\b u\right)\b v\ds$.
\end{remark}

In the next remark, we outline how we could replace Assumptions \ref{a:ElastObjectiveFctl} and \ref{a:ElastDiffAss} by weaker assumptions. In order to simplify the estimates in the following analysis we prefer \ref{a:ElastObjectiveFctl} and \ref{a:ElastDiffAss}.
\begin{remark}
	We could generalise the results to objective functionals satisfying
	\begin{align}\label{e:ElastPhaseObjFctGrowFAlt}\left|h_\Omega\left(x,\b v\right)\right|\leq a_1(x)+b_1(x)\left|\b v\right|^p,\quad\forall\b v\in\R^d, \text{ a.e. }x\in\Omega\end{align}
	for some functions $a_1\in L^1(\Omega)$ and $b_1\in L^\infty(\Omega)$, instead of requiring $\eqref{e:ElastPhaseObjFctGrowF}$. Here, $p\geq2$ has to be chosen such that $\b H^1(\Omega)\hookrightarrow \b L^p(\Omega)$ is a compact imbedding, hence $2\leq p<\infty$ for $d=2$ and $2\leq p<6$ for $d=3$. We then obtain that $L^p(\Omega)^d\ni\b v\mapsto h_\Omega\left(\cdot,\b v\right)\in L^1(\Omega)$ is well-defined and continuous and all proofs can be adapted. In this case, we have to replace $\eqref{e:ElastPhaseObjFctGrowFDer}$ in Assumption \ref{a:ElastDiffAss} by
	$
	|\der_{2}h_\Omega\left(x,\b v\right)|\leq \hat a_1(x)+\hat b_1(x)|\b v|^{p-1}$ for all $\b v\in\R^d$ and a.e. $x\in\Omega$ where $\hat a_1\in L^{\nicefrac{p}{p-1}}(\Omega)$, $\hat b_1\in L^\infty(\Omega)$ to obtain that $\b L^p\left(\Omega\right)\ni\b u\mapsto\int_\Omega h_\Omega\left(x,\b u(x)\right)\dx$
	is continuously Fr\'{e}chet differentiable.  The same holds for the choice of $h_\Gamma$.\\
\end{remark}

In order to obtain a well-posed problem we add to the cost functional $\mathbb H$ in $\eqref{e:ObjFctlOneReg}$ a regularization term. In the sharp interface problem a multiple of the perimeter of the free boundary between the two materials is used. The exact definition of the perimeter is introduced now. Since we describe the sharp interface model by a design variable $\varphi:\Omega\to\left\{\pm1\right\}$, where $\left\{\varphi=\pm1\right\}$ describe the two different materials, this design variable is going to be a function of bounded variation. We give here a brief introduction in the notation of Caccioppoli sets and functions of bounded variations, but for a detailed introduction we refer to \cite{ambrosio,evans_gariepy}. We call a function $\varphi\in L^1(\Omega)$ a function of bounded variation if its distributional derivative is a vector-valued finite Radon measure. The space of functions of bounded variation in $\Omega$ is denoted by $BV(\Omega)$, and by $BV(\Omega,\{\pm1\})$ we denote functions in $BV(\Omega)$ having only the values $\pm1$ a.e. in $\Omega$. We then call a measurable set $E\subset\Omega$ Caccioppoli set if $\chi_E\in BV(\Omega)$. For any Caccioppoli set $E$, one can hence define the total variation $\left|\der\chi_E\right|(\Omega)$ of $\der\chi_E$, as $\der\chi_E$ is a finite measure. This value is then called the perimeter of $E$ in $\Omega$ and is denoted by $P_\Omega\left(E\right):=\left|\der\chi_E\right|(\Omega)$.\\

We include additionally a volume constraint in the optimization problem. By assuming that the design variable $\varphi$ fulfills $\int_\Omega\varphi\dx\leq\beta|\Omega|$, which is equivalent to $|\{\varphi=1\}|\leq\tfrac{(\beta+1)}{2}|\Omega|$ and $|\{\varphi=-1\}|\geq\tfrac{(1-\beta)}{2}|\Omega|$, for some fixed constant $\beta\in(-1,1)$ we prescribe a maximal amount of the material corresponding to $\{\varphi=1\}$ (and thus a minimal amount of $\{\varphi=-1\}$) that can be used during the optimization process.\\
Our admissible design variables for the sharp interface problem hence are chosen in the set
\begin{align}\Phi_{ad}^0:=\left\{\varphi\in BV\left(\Omega,\left\{\pm1\right\}\right)\mid\int_\Omega\varphi\dx\leq\beta|\Omega|\right\}.\end{align}
In the phase field formulation of the shape optimization problem we approximate the perimeter by the Ginzburg-Landau energy  
\begin{align}\label{e:GinzburgLandau}E_\epsilon(\varphi):=\int_\Omega\frac\epsilon2|\nabla\varphi|^2+\frac1\epsilon\psi(\varphi)\dx\end{align}
with a double obstacle potential $\psi:\R\to\overline\R:=\R\cup\{\infty\}$ given by
\begin{align}\label{e:Potential}\psi\left(\varphi\right):=\begin{cases}\psi_0\left(\varphi\right),&\text{if }\left|\varphi\right|\leq1\\+\infty,&\text{if }\left|\varphi\right|>1\end{cases},\qquad \psi_0\left(\varphi\right):=\frac12\left(1-\varphi^2\right).\end{align}

\noindent The functionals $(E_\epsilon)_{\epsilon>0}$ $\Gamma$-converge in $L^1(\Omega)$ to $\varphi\mapsto c_0 P_\Omega\left(\{\varphi=1\}\right)$ with $c_0:=\int_{-1}^1\sqrt{2\psi(s)}\ds=\frac\pi2$ as $\epsilon\searrow0$, see for instance \cite{modica,modica_mortola}.\\

In the phase field setting, the design variable $\varphi$ is allowed to have values in $[-1,1]$ and thus there may be a transition area between the areas $\{\varphi=-1\}$ and $\{\varphi=1\}$. The admissible set in the phase field setting is given by

\begin{align}\Phi_{ad}:=\left\{\varphi\in H^1(\Omega)\mid\int_\Omega\varphi\dx\leq\beta|\Omega|,\,\left|\varphi\right|\leq1\text{ a.e. in }\Omega\right\}\end{align}
and the extended admissible set by $\overline\Phi_{ad}:=\left\{\varphi\in H^1(\Omega)\mid\left|\varphi\right|\leq1\text{ a.e. in }\Omega\right\}.$

\begin{remark}\label{r:EqualityConstraint}
Instead of $\int_\Omega\varphi\dx\leq\beta|\Omega|$ we could also use an equality constraint of the form $\int_\Omega\varphi\dx=\beta|\Omega|$. This prescribes then the exact volume fraction of each material in advance. In this setting, the same analysis can be carried out. The results whereon our sharp interface analysis is based only deal with an equality constraint, see for instance \cite{blowey_elliot, garcke_paper, modica}.
\end{remark}

	As we derive first order optimality conditions by varying the free boundary between the two materials with transformations, we introduce here the admissible transformations and its corresponding velocity fields:
	
	\newcounter{VCount}
\renewcommand{\theVCount}{\textbf{(V{\arabic{VCount}})}}
\newcounter{VListCount}   
	\begin{definition}[$\mathcal V_{ad}$, $\mathcal T_{ad}$] \label{d:AdmissibleTransVel}
		 The space $\mathcal V_{ad}$ of admissible velocity fields is defined as the set of  all $V\in C\left(\left[-\tau,\tau\right]\times\overline\Omega,\R^d\right)$, where $\tau>0$ is some fixed, small constant, such that it holds:
		\begin{list}{\theVCount}{\usecounter{VCount}}
 \item\label{l:AssVSmooth}
 	\begin{description}\item[\textbf{(V1a)}]
		$V(t,\cdot)\in C^2\left(\overline\Omega,\R^d\right)$, \item[\hspace{0.0cm}\textbf{(V1b)}]$\exists C>0$: $\left\|V\left(\cdot,y\right)-V\left(\cdot,x\right)\right\|_{C\left(\left[-\tau,\tau\right],\R^d\right)}\leq C\left|x-y\right|\,\,\forall x,y\in\overline\Omega$,
		\end{description}
		\item\label{l:AssVNormal}
		$V(t,x)\cdot\b n(x)=0$ for all $x\in\partial\Omega$,
		\item\label{l:AssVBoundaryG} $V(t,x)=\b0$ for every $x\in\Gamma_D$.
		\setcounter{VCount}{\value{VCount}}
		\setcounter{VListCount}{\value{VCount}}
		\end{list}
		Then the space $\mathcal T_{ad}$ of admissible transformations is defined as solutions of the ordinary differential equation
		\begin{equation}\label{e:ODE}\partial_t T_t(x)=V(t,T_t(x)),\qquad T_0(x)=x\end{equation}
		with $V\in \mathcal V_{ad}$,
		which gives some $T:\left(-\tilde\tau,\tilde\tau\right)\times\overline\Omega\to\overline\Omega$, with $0<\tilde\tau$ small enough.\\
		We often use the notation $V(t)=V(t,\cdot)$.
		\end{definition}
		
			\begin{remark}\label{r:NotAssPropertiesOfTransformations}
		Let $V\in{\mathcal V}_{ad}$ and $T\in{\mathcal V}_{ad}$ be the transformation associated to $V$ by $\eqref{e:ODE}$. Then $T_t:\overline\Omega\to\overline\Omega$ is bijective and $T(\cdot, x)\in C^1((-\tau,\tau),\R^d)$ for all $x\in\overline\Omega$ and $\tau>0$ small enough. These and other properties are discussed in detail in \cite{delfour, deflourpaper}.
	\end{remark}

We finish this introduction by two typical examples which are commonly used as objective functionals in structural optimization. For a deeper discussion on those problems and some further applications we refer for instance to \cite{bendsoe2003topology}.

\begin{bsp}[Mean compliance]\label{ex:MeanComp}
One commonly used objective in structural optimization is the minimization of the mean compliance, which is for a structure in its equilibrium configuration given by
$\int_\Omega \b f\cdot\b u\dx+\int_{\Gamma_g}\b g\cdot\b u\ds.$ The aim of minimizing this objective functional can be interpreted as maximizing the stiffness under the given forces or as minimizing the stored mechanical energy. We notice, that this is equivalent to minimizing 
$\int_\Omega\C\left(\varphi\right)\left(\Epsilon\left(\b u\right)-\overline\Epsilon\left(\varphi\right)\right):\Epsilon\left(\b u\right)\dx$
if $\b u$ solves the state equations $\eqref{e:ElastPhaseFirstState}$. 
\end{bsp}

\begin{bsp}[Compliant mechanism]\label{ex:CompMech}
The typical compliant mechanism objective functional used in topology optimization is given by the tracking type functional
$\frac12\int_\Omega c\left|\b u-\b u_\Omega\right|^2\dx$
where $\b u_\Omega\in\b L^2(\Omega)$ is some desired displacement, and $c\in L^\infty(\Omega)$, $c\geq 0$, is a weighting factor.
\end{bsp}

In the following, by minimizers we always mean global minimizers.

\subsection{Perimeter penalized shape optimization problem}\label{s:ElastSharpProblemForm}

The sharp interface problem that we consider in this section is given by

\begin{equation}\begin{split}\label{e:ElastObjFctSharp}
\min_{(\varphi,\b u)} J_0\left(\varphi,\b u\right)&:=\mathbb H(\b u)+\gamma c_0P_\Omega\left(\left\{\varphi=1\right\}\right)=\\
&=\int_\Omega h_\Omega\left(x,\b u\right)\dx+\int_{\Gamma_g}h_\Gamma\left(s,\b u\right)\ds+\gamma c_0P_\Omega\left(\left\{\varphi=1\right\}\right)\end{split}
\end{equation}
with
$\left(\varphi,\b u\right)\in\Phi_{ad}^0\times\b H^1_D(\Omega)$ such that $\eqref{e:ElastPhaseFirstState}$ holds, i.e.
\begin{align}\label{e:ElastSharpStateWeak}\int_\Omega\C\left(\varphi\right)\left(\Epsilon\left(\b u\right)-\overline\Epsilon\left(\varphi\right)\right):\Epsilon\left(\b v\right)\dx=\int_\Omega\b f\cdot\b v\dx+\int_{\Gamma_g}\b g\cdot\b v\ds\quad\forall\b v\in\b H^1_D(\Omega).\end{align}

This is a topology and shape optimization problem, where $\varphi\in\Phi_{ad}^0=\{\varphi\in BV\left(\Omega,\{\pm1\}\right)\mid\int_\Omega\varphi\dx\leq\beta|\Omega|\}$ plays the role of the design variable, and can only have the discrete values $\pm1$. The perimeter in the cost functional ensures the existence of a minimizer where the weighting factor $\gamma>0$ can be arbitrary.
In the remainder of this subsection we summarize often results where the proofs are given later in this paper or in some previous work. Studying the reduced objective functional $j_0:L^1(\Omega)\to\overline\R$, 
$$j_0(\varphi):=\begin{cases}J_0\left(\varphi,\b S\left(\varphi\right)\right), &\text{if }\varphi\in\Phi_{ad}^0,\\ +\infty,&\text{otherwise,}\end{cases}$$ 
we obtain by using the direct method in the calculus of variations the well-posedness of the optimization problem:

\begin{theorem}\label{t:SharpExistMin} Under the assumptions \ref{a:ElastOmega}-\ref{a:ElastObjectiveFctl}, there exists at least one minimizer of $\eqref{e:ElastObjFctSharp}-\eqref{e:ElastSharpStateWeak}$.
\end{theorem}
\begin{proof}
	We use that $\eqref{e:ElastObjFctSharp}-\eqref{e:ElastSharpStateWeak}$ is equivalent to $\min_{\varphi\in L^1(\Omega)} j_0(\varphi)$. According to Assumption \ref{a:ElastObjectiveFctl}, the objective functional $j_0$ is bounded from below and hence we may choose a minimizing sequence $(\varphi_k)_{k\in\N}$ for $j_0$. Thus $P_\Omega(\{\varphi_k=1\})=|\der\varphi_k|(\Omega)$ is uniformly bounded. We obtain therefrom and the fact that $\|\varphi_k\|_{L^\infty(\Omega)}\leq 1$ for all $k\in\N$ that $(\varphi_k)_{k\in\N}$ is uniformly bounded in $BV(\Omega)$. As $BV(\Omega)$ imbeds compactly into $L^1(\Omega)$, we hence find that $(\varphi_k)_{k\in\N}$ has a subsequence $(\varphi_{k_l})_{l\in\N}$ converging in $L^1(\Omega)$ and pointwise to some limit element $\varphi \in BV(\Omega)$. From the pointwise convergence we obtain $\varphi(x)\in\{\pm1\}$ for almost every $x\in\Omega$ and the convergence in $L^1(\Omega)$ yields directly $\int_\Omega\varphi\dx\leq\beta|\Omega|$. Hence we have $\varphi\in \Phi_{ad}^0$. From Lemma~\ref{l:ElastConvGammaConvFEcont}, which is given in Section~\ref{s:ProofConvRes}, we obtain that $\varphi\mapsto\int_\Omega h_\Omega(x,\b S(\varphi))\dx+\int_{\Gamma_g}h_\Gamma(s,\b S(\varphi))\ds$ is continuous in $L^1(\Omega)$. Moreover, the perimeter functional $\varphi\mapsto P_\Omega(\{\varphi=1\})$ is lower semicontinuous in $L^1(\Omega)$, see \cite{ambrosio}, and thus we obtain $j_0(\varphi)\leq\liminf_{l\to\infty} j_0(\varphi_{k_l})$. This shows that $\varphi$ is a minimizer of $j_0$, and hence $(\varphi,\b S(\varphi))$ is a minimizer of $\eqref{e:ElastObjFctSharp}-\eqref{e:ElastSharpStateWeak}$.	
\end{proof}

Our next aim is to deduce first order necessary optimality conditions. For this purpose, we use the ideas of shape calculus, which means we apply geometric variations. As already mentioned in the introduction we do not impose any additional regularity assumption on the minimizing set. We obtain:

\begin{theorem}\label{t:ElastGeomVarSharp} Assume \ref{a:ElastOmega}-\ref{a:ElastDiffAss}. Then for any minimizer $\left(\varphi_0,\b u_0\right)\in \Phi_{ad}^0\times\b H^1_D(\Omega)$ of $\eqref{e:ElastObjFctSharp}-\eqref{e:ElastSharpStateWeak}$ the following necessary optimality conditions hold: There exists a Lagrange multiplier $\lambda_0\geq0$ for the integral constraint such that
	\begin{align}\label{e:ElastSharpOptCond1}\partial_t|_{t=0}j_0\left(\varphi_0\circ T_t^{-1}\right)=-\lambda_0\int_\Omega\varphi_0\div V(0)\dx,\quad\lambda_0\left(\int_\Omega\varphi_0\dx-\beta\left|\Omega\right|\right)=0
	\end{align}
	holds for all $T\in{\mathcal T}_{ad}$ with corresponding velocity $V\in{\mathcal V}_{ad}$, where the derivative is given by the following formula:
		\begin{equation}\label{e:ElastSharpPartTJ0E}\begin{split}
		\partial_t|_{t=0}j_0\left(\varphi_0\circ T_t^{-1}\right)&=\partial_t|_{t=0}\mathbb H\left(\b S\left(\varphi_0\circ T_t^{-1}\right)\right)+\gamma c_0\int_\Omega\left(\div V(0)-\nu\cdot\nabla V(0)\nu\right)\,\mathrm d\left|\der\chi_{E_0}\right|
	\end{split}\end{equation}
and
	
	\begin{equation}\label{e:ElastSharpPartTJ0E2}\begin{split}
		&\partial_t|_{t=0}\mathbb H\left(\b S\left(\varphi_0\circ T_t^{-1}\right)\right)=\int_\Omega\left[\der h_\Omega\left(x,\b u_0\right)\left(V(0),\dot{\b u}_0\left[V\right]\right)+h_\Omega\left(x,\b u_0\right)\div V(0)\right]\dx+\\
		&+\int_{\Gamma_g}\left[\der h_\Gamma\left(s,\b u_0\right)\left(V(0),\dot{\b u}_0\left[V\right]\right)+h_\Gamma\left(s,\b u_0\right)\left(\div V(0)-\b n\cdot\nabla V(0)\b n\right)\right]\ds
	\end{split}\end{equation}
	with $\nu:=\frac{\der\chi_{E_0}}{|\der\chi_{E_0}|}$ being the generalised unit normal on the Caccioppoli set $E_0:=\{\varphi_0=1\}$, compare \cite{ambrosio}. Moreover, $\dot{\b u}_0\left[V\right]\in\b H^1_D(\Omega)$ is given as the solution of

\begin{equation}\label{e:ElastPhaseDotUEpsilonEquationWeak}\begin{split}&\int_\Omega\C(\varphi_0)\Epsilon(\dot{\b u}_0[V]):\Epsilon(\b v)\dx=\int_\Omega\C(\varphi_0)\frac12\left(\nabla V(0)\nabla\b u_0+\left(\nabla V(0)\nabla\b u_0\right)^T\right):\Epsilon(\b v)+\\
&+\C(\varphi_0)\left(\Epsilon(\b u_0)-\overline\Epsilon\left(\varphi_0\right)\right):\frac12\left(\nabla V(0)\nabla\b v+\left(\nabla V(0)\nabla\b v\right)^T\right)-\\
&-\C(\varphi_0)\left(\Epsilon(\b u_0)-\overline\Epsilon\left(\varphi_0\right)\right):\Epsilon(\b v)\div V(0)\dx-\int_\Omega\b f\cdot\der \b vV(0)\dx-\int_{\Gamma_g}\b g\cdot\der\b vV(0)\ds\end{split}\end{equation}

which has to hold for all $\b v\in\b H^1_D(\Omega)$.
\end{theorem}

The proof of this theorem can be found in Section \ref{s:ProofOptcond}.

\begin{remark}\label{r:ElastSharpGammgRegualirSurfDiv}
	If we assume that $\Gamma_g$ has $C^2$-regularity, $\eqref{e:ElastSharpPartTJ0E2}$ can be rewritten into a more convenient form by using the identity $\div_{\Gamma_g}V(0)=\div V(0)-\b n\cdot\nabla V(0)\b n$ on $\Gamma_g$.

\end{remark}

We can now reformulate those optimality conditions under more regularity assumptions on the minimizing set $E_0=\{\varphi_0=1\}$ and the given data. In particular, we can then compare our results to those obtained in literature, see Remark \ref{r:compare}.

\begin{theorem}\label{t:ElastSharpSmoothReform} Assume \ref{a:ElastOmega}-\ref{a:ElastDiffAss}.	Let $\left(\varphi_0,\b u_0\right)\in\Phi_E^0\times\b H^1_D(\Omega)$ be minimizers of $\eqref{e:ElastObjFctSharp}-\eqref{e:ElastSharpStateWeak}$. Assume there are open sets $\Omega_1,\Omega_2\subset\Omega$ such that $\varphi_0=1$ a.e. on $\Omega_1$ and $\varphi_0=-1$ a.e. in $\Omega_2$. Let $\b g\in\b H^{\frac12}\left(\partial\Omega\right)$ and the objective functional is assumed to be chosen in such a way that $\der_2h_\Omega\left(\cdot,\b u\right)\in\b L^2(\Omega)$ and $\der_2h_\Gamma\left(\cdot,\b u\right)\in\b H^{\frac12}(\Gamma_g)$ for all $\b u\in\b H^1(\Omega)$. If $h_\Gamma\left(\cdot,\b u_0\left(\cdot\right)\right)\not\equiv0$, we assume additionally that $\Gamma_g$ has $C^2$-regularity. Assume that $\Gamma_0:=\partial\Omega_1\cap\partial\Omega_2\in C^2$ and $d\left(\Gamma_0,\partial\Omega\right)>0$. By 
	$\left[\b w\right]_{\Gamma_0}(x):=\b w|_{\Omega_1}(x)-\b w|_{\Omega_2}(x)$	we denote the jump of $\b w$ along the interface $\Gamma_0$, and $\nu$ is the outer unit normal on $\Omega_1$. Let $\kappa=\div_{\Gamma_0}\nu$ be the mean curvature of $\Gamma_0$. Then the optimality conditions derived in Theorem~\ref{t:ElastGeomVarSharp} are equivalent to the following system:
	
\begin{equation}\label{e:ElastSharpSmoothReformOptCondOnGamma}\left.
	\begin{aligned}
\gamma c_0\kappa-&\left[\C(\varphi_0)\left(\Epsilon\left(\b u_0\right)-\overline\Epsilon\left(\varphi_0\right)\right):\Epsilon\left(\b q_0\right)\right]_{\Gamma_0}+\\
&+\left[\C(\varphi_0)\left(\Epsilon\left(\b u_0\right)-\overline\Epsilon\left(\varphi_0\right)\right)\nu\cdot\partial_{\nu}\b q_0\right]_{\Gamma_0}+\\
&+\left[\C(\varphi_0)\Epsilon\left(\b q_0\right)\nu\cdot\partial_{\nu}\b u_0\right]_{\Gamma_0}+2\lambda_0+\left[h_\Omega\left(x,\b u_0\right)\right]_{\Gamma_0}=0\quad\text{ on }\Gamma_0
\end{aligned}\right\}\end{equation}

	\begin{equation}\label{e:ElastSharpLagragne}
\lambda_0\left(\int_\Omega\varphi_0\dx-\beta\left|\Omega\right|\right)=0,\,\,\lambda_0\geq0,\,\,\int_\Omega\varphi_0\dx\leq\beta\left|\Omega\right|,\end{equation}	
together with the state equation $\eqref{e:ElastPhaseFirstState}$ connecting $\varphi_0$ and $\b u_0$. Here, the adjoint variable $\b q_0\in\b H^1_D(\Omega)$ is the unique solution of the adjoint equation
\begin{equation}\label{e:Adjoint}\int_\Omega\C(\varphi_0)\Epsilon\left(\b q_0\right):\Epsilon\left(\b v\right)\dx=\int_\Omega\der_2h_\Omega\left(x,\b u_0\right)\b v\dx+\int_{\Gamma_g}\der_2h_\Gamma\left(s,\b u_0\right)\b v\ds\quad\forall\b v\in\b H^1_D(\Omega).\end{equation}

\end{theorem}
The statement of Theorem \ref{t:ElastSharpSmoothReform} can be shown by using basic calclations such as integration by parts and chain rule. This calculation has been carried out for instance in \cite[Theorem 25.3]{hecht}.

\begin{remark}\label{r:compare}
	  In \cite{relatingphasefield} the same optimality system for the sharp interface setting has been derived from the phase field model by formally matched asymptotics for the mean compliance and compliant mechanism problems mentioned in Examples~\ref{ex:MeanComp} and~\ref{ex:CompMech}. However, no eigenstrain has been taken into account in \cite{relatingphasefield} while this is included in the above problem setting. Applying shape sensitivity analysis yields the same result as we have found in Theorem~\ref{t:ElastSharpSmoothReform} (see e.g. \cite{allairemulti,allaire2010damage,sturm}).
\end{remark}

\subsection{Phase field approximation}\label{s:ElastPhase}
In a diffuse interface setting in terms of a phase field formulation the shape and topology optimization problem of finding the optimal material distribution of two given materials is given by
\begin{equation}\begin{split}\label{e:ElastObjFctPhase}
\min_{(\varphi,\b u)} J_\epsilon\left(\varphi,\b u\right)&:=\mathbb H\left(\b u\right)+\gamma E_\epsilon\left(\varphi\right)=\\
&=\int_\Omega h_\Omega\left(x,\b u\right)\dx+\int_{\Gamma_g} h_\Gamma\left(s,\b u\right)\ds+\gamma\int_\Omega\frac{\epsilon}{2}\left|\nabla\varphi\right|^2+\frac1\epsilon\psi\left(\varphi\right)\dx
\end{split}\end{equation}
with
$\left(\varphi,\b u\right)\in\Phi_{ad}\times\b H^1_D(\Omega)$ such that $\eqref{e:ElastPhaseFirstState}$ holds, i.e.
\begin{align}\label{e:ElastPhaseStateWeak}\int_\Omega\C\left(\varphi\right)\left(\Epsilon\left(\b u\right)-\overline\Epsilon\left(\varphi\right)\right):\Epsilon\left(\b v\right)\dx=\int_\Omega\b f\cdot\b v\dx+\int_{\Gamma_g}\b g\cdot\b v\ds\quad\forall\b v\in\b H^1_D(\Omega).\end{align}
Hence the design variable in this problem is given by $\varphi\in\Phi_{ad}=\{\varphi\in H^1(\Omega)\mid\int_\Omega\varphi\dx\leq\beta|\Omega|,\,\left|\varphi\right|\leq1\text{ a.e. in }\Omega\}$. The regions filled with material one or two are represented by $\left\{x\in\Omega\mid\varphi(x)=1\right\}$ and $\left\{x\in\Omega\mid\varphi(x)=-1\right\}$, respectively. The design variable $\varphi$ is also allowed to take values between minus one and one, which leads to a small transitional area whose thickness is proportional to a small parameter $\epsilon>0$. Thus, as $\epsilon$ tends to zero, we will arrive in a sharp interface problem and the interfacial layer vanishes.\\
As was already discussed in Section \ref{s:NotAss}, the Ginzburg-Landau energy $E_\epsilon$, compare $\eqref{e:GinzburgLandau}$, appearing in the objective functional is essential for the existence of a minimizer and $(E_\epsilon)_\epsilon$ $\Gamma$-converge to a multiple of the perimeter functional as $\epsilon$ tends to zero, cf. \cite{modica,modica_mortola}. The parameter $\gamma>0$ is an arbitrary fixed constant and can be considered as a weighting factor of the perimeter penalization.\\

%The last two terms of the objective functional form the Ginzburg-Landau energy $E_\epsilon$ multiplied by $\gamma$, compare $\eqref{e:GinzburgLandau}$. This functional is essential for the existence of a minimizer of the problem and it $\Gamma$-converges to a multiple of the perimeter functional as $\epsilon$ tends to zero, cf. \cite{modica,modica_mortola}. The parameter $\gamma>0$ is an arbitrary fixed constant and can be considered as a weighting factor of the perimeter penalization.\\

We know from Lemma \ref{l:SolOp} that there is a solution operator $\b S$ for the constraints $\eqref{e:ElastPhaseStateWeak}$. Thus, we can reformulate the optimization problem into $\min_{\varphi\in L^1(\Omega)}j_\epsilon(\varphi)$ where $j_\epsilon:L^1(\Omega)\to\overline\R$, 

$$j_\epsilon(\varphi):=\begin{cases}J_\epsilon\left(\varphi,\b S\left(\varphi\right)\right), &\text{if }\varphi\in\Phi_{ad},\\ +\infty,&\text{otherwise,}\end{cases}$$ 
is the reduced objective functional. We start by discussing the optimization problem $\eqref{e:ElastObjFctPhase}-\eqref{e:ElastPhaseStateWeak}$ in view of well-posedness for fixed $\epsilon>0$.
\begin{theorem}\label{t:ElastPhaseExistMin}
	Under the assumptions \ref{a:ElastOmega}-\ref{a:ElastObjectiveFctl}, there exists at least one minimizer of 
$\eqref{e:ElastObjFctPhase}-\eqref{e:ElastPhaseStateWeak}$.\\
\end{theorem}
{\em Sketch of a proof}. Proof of the previous theorem
This is established by using the direct method in the calculus of variations. For this purpose, we follow the lines of the proof of Theorem \ref{t:SharpExistMin}. Instead of using the compact imbedding $BV(\Omega)\hookrightarrow L^1(\Omega)$ we use here that $H^1(\Omega)$ embeds compactly into $L^2(\Omega)$ and replace the lower semicontinuity of the perimeter functional by the lower semicontinuity of the Ginzburg-Landau energy with respect to convergence in $L^2(\Omega)$. For more details we refer to \cite[Theorem 24.1]{hecht} or \cite{relatingphasefield}, where the case with $\overline\Epsilon\equiv0$ is treated.
\qquad$\square$

We obtain optimality conditions by geometric variations:

\begin{theorem}\label{t:PhseOptCond} Assume \ref{a:ElastOmega}-\ref{a:ElastDiffAss}. Then for any minimizer $\left(\varphi_\epsilon,\b u_\epsilon\right)$ of $\eqref{e:ElastObjFctPhase}-\eqref{e:ElastPhaseStateWeak}$ the following necessary optimality conditions hold: There exists a Lagrange multiplier $\lambda_\epsilon\geq0$ for the integral constraint such that
	
	\begin{align}\label{e:ElastPhaseOptCond1}\partial_t|_{t=0}j_\epsilon\left(\varphi_\epsilon\circ T_t^{-1}\right)=-\lambda_\epsilon\int_\Omega\varphi_\epsilon\div V(0)\dx,\quad\lambda_\epsilon\left(\int_\Omega\varphi_\epsilon\dx-\beta\left|\Omega\right|\right)=0
	\end{align}
	holds for all $T\in{\mathcal T}_{ad}$ with corresponding velocity $V\in{\mathcal V}_{ad}$. The derivative is given by the following formula:
		\begin{equation}\begin{split}
	\label{e:ElastPhaseBoundaryVariationJEpsilonDifTerm}	&\partial_t|_{t=0}j_\epsilon\left(\varphi_\epsilon\circ T_t^{-1}\right)=\partial_t|_{t=0}\mathbb H\left(\b S\left(\varphi_\epsilon\circ T_t^{-1}\right)\right)+\\	&+\int_\Omega\left(\frac{\gamma\epsilon}{2}\left|\nabla\varphi_\epsilon\right|^2+\frac\gamma\epsilon\psi\left(\varphi_\epsilon\right)\right)\div V(0)-\gamma\epsilon\nabla\varphi_\epsilon\cdot\nabla V(0)\nabla\varphi_\epsilon\dx
	\end{split}\end{equation}

where $\dot{\b u}_\epsilon\left[V\right]\in\b H^1_D(\Omega)$ is given as the solution of $\eqref{e:ElastPhaseDotUEpsilonEquationWeak}$ with $\varphi_0$ replaced by $\varphi_\epsilon$ and $\b u_0$ by $\b u_\epsilon$. The exact formula for $\partial_t|_{t=0}\mathbb H\left(\b S\left(\varphi_\epsilon\circ T_t^{-1}\right)\right)$ is given in $\eqref{e:ElastSharpPartTJ0E2}$.\\
\end{theorem}
The proof can be found in Section \ref{s:ProofOptcond}.

\begin{remark}\label{r:VarInequality}
	One can also consider $\eqref{e:ElastObjFctPhase}-\eqref{e:ElastPhaseStateWeak}$ in the framework of optimal control problems. By parametric variations one then obtains as necessary optimality conditions the following variational inequality:
	\begin{equation}\begin{split}\label{e:ElastVarInequality}j'_\epsilon(\varphi_\epsilon)(\varphi-\varphi_\epsilon)+\lambda_\epsilon\int_\Omega\left(\varphi-\varphi_\epsilon\right)\dx\geq 0\quad\forall\varphi\in\overline{\Phi}_{ad}\end{split}\end{equation}
	where
		\begin{equation}\label{e:ReducedObjFctlDirDer}\begin{aligned}
		&j'_\epsilon(\varphi_\epsilon)(\varphi-\varphi_\epsilon)=\left(\gamma\epsilon\nabla\varphi_\epsilon,\nabla\left(\varphi-\varphi_\epsilon\right)\right)_{\b L^2\left(\Omega\right)}+\\
		&+\left(\frac\gamma\epsilon\psi'_0\left(\varphi_\epsilon\right)+\left(\C(\varphi_\epsilon)\overline\Epsilon'\left(\varphi_\epsilon\right)-\C'\left(\varphi_\epsilon\right)\left(\Epsilon\left(\b u_\epsilon\right)-\overline\Epsilon\left(\varphi_\epsilon\right)\right)\right):\Epsilon\left(\b q_\epsilon\right),\varphi-\varphi_\epsilon\right)_{L^2\left(\Omega\right)}\end{aligned}\end{equation}
	is the directional derivative of the reduced objective functional $j_\epsilon$. The variational inequality $\eqref{e:ElastVarInequality}$ has to be fulfilled	together with the state equations $\eqref{e:ElastPhaseStateWeak}$ and the adjoint equation $\eqref{e:Adjoint}$. This approach can for instance be used for numerical methods. More details on these optimality criteria can be found in \cite{relatingphasefield, hecht}.
\end{remark}

\subsection{Sharp interface limit}\label{s:Limit}

In this section we give the results on relating the phase field problems introduced in Section \ref{s:ElastPhase} to the sharp interface formulation, which was discussed in Section \ref{s:ElastSharpProblemForm}. % For this purpose, we will on the one hand show that the reduced objective functionals $\left(j_\epsilon\right)_{\epsilon>0}$ $\Gamma$-converge in $L^1(\Omega)$ to the reduced objective functional $j_0$ describing the sharp interface optimization problem. On the other hand we will show that the optimality systems of the phase field model, which were obtained by geometric variations, approximate an optimality system of the sharp interface problem, too.\\

\begin{theorem}\label{t:ElastConvGammaConv}
	Under the assumptions \ref{a:ElastOmega}-\ref{a:ElastObjectiveFctl}, the functionals $\left(j_\epsilon\right)_{\epsilon>0}$ $\Gamma$-converge in $L^1(\Omega)$ to $j_0$ as $\epsilon\searrow0$.\\ 
\end{theorem}
The proof of this theorem is given in Section~\ref{s:ProofConvRes}. As a consequence, we obtain directly:
\begin{corollary}\label{c:ConvMinElastConv} Assume \ref{a:ElastOmega}-\ref{a:ElastObjectiveFctl}. Let $\left(\varphi_\epsilon\right)_{\epsilon>0}$ be minimizers of $\left(j_\epsilon\right)_{\epsilon>0}$. Then there exists a subsequence, denoted by the same, and an element $\varphi_0\in L^1(\Omega)$ such that $\lim_{\epsilon\searrow0}\|\varphi_\epsilon-\varphi_0\|_{L^1(\Omega)}=0$. Besides, $\varphi_0$ is a minimizer of $j_0$ and it holds
	 $\lim_{\epsilon\searrow0}j_\epsilon\left(\varphi_\epsilon\right)=j_0\left(\varphi_0\right).$
\end{corollary}
\begin{proof}
From $\sup_{\epsilon>0}j_\epsilon(\varphi_\epsilon)<\infty$ we find $\sup_{\epsilon>0}\int_\Omega\left(\frac\epsilon2|\nabla\varphi_\epsilon|^2+\frac1\epsilon\psi(\varphi_\epsilon)\right)\dx<\infty$. We can apply the compactness argument \cite[Proposition 3]{modica} to find a subsequence of $(\varphi_\epsilon)_{\epsilon>0}$ converging in $L^1(\Omega)$ to some element $\varphi_0$ as $\epsilon\searrow0.$ Then the previous theorem and standard results for $\Gamma$-convergence, see for instance \cite{dalmaso}, yield the assertion.
\end{proof}

\begin{theorem}\label{t:ConvOPtSys} Assume \ref{a:ElastOmega}-\ref{a:ElastDiffAss}.
	Let $\left(\varphi_\epsilon\right)_{\epsilon>0}$ be minimizers of $\left(j_\epsilon\right)_{\epsilon>0}$. Then there exists a subsequence, which is denoted by the same, that converges in $L^1(\Omega)$ to a minimizer $\varphi_0$ of $j_0$. Moreover, it holds $\lim_{\epsilon\searrow0}\partial_t|_{t=0} j_\epsilon\left(\varphi_\epsilon\circ T_t^{-1}\right)=\partial_t|_{t=0}j_0\left(\varphi_0\circ T_t^{-1}\right)$ for all $T\in{\mathcal T}_{ad}$. If $\left|\left\{\varphi_0=1\right\}\right|>0$	then we additionally have the following convergence results:
	\begin{subequations}\begin{align}
	\label{e:ElastConvOptSysUEpsilonConv}	\b u_\epsilon \stackrel{\epsilon\searrow0}{\rightharpoonup}\b u_0, & \quad \dot{\b u}_\epsilon\left[V\right]\stackrel{\epsilon\searrow0}{\rightharpoonup} \dot{\b u}_0\left[V\right]  && \text{in }\b H^1(\Omega),\\	
		\lambda_\epsilon\stackrel{\epsilon\searrow0}{\longrightarrow} \lambda_0, & \quad j_\epsilon(\varphi_\epsilon)\stackrel{\epsilon\searrow0}{\longrightarrow}j_0(\varphi_0)&&\text{in }\R,
	\end{align}\end{subequations}
	where $\b u_\epsilon=\b S(\varphi_\epsilon)$, $\b u_0=\b S(\varphi_0)$ and $\left(\lambda_\epsilon\right)_{\epsilon>0}\subseteq\R^+_0$ are Lagrange multipliers for the integral constraint defined in Theorem~\ref{t:PhseOptCond}, $\lambda_0\in\R^+_0$ is a Lagrange multiplier for the integral constraint in the sharp interface setting since it fulfills $\eqref{e:ElastSharpOptCond1}$.\\
\end{theorem}
The proof is given in Section \ref{s:ProofConvRes}.

\section{Numerical experiments}\label{s:Numerics}
We want to verify the reliability and practical relevance of the phase field approximation by means of numerical experiments. Besides, we also study the behaviour for decreasing phase field parameters $\epsilon>0$ and see that the convergence results stated in the previous section are also indicated by numerics. For the numerics, the admissible design functions are chosen in
$\widetilde{\Phi}_{ad}:=\{\varphi\in H^1(\Omega)\mid \int_\Omega\varphi\dx=\beta\left|\Omega\right|,\,|\varphi|\leq 1\text{ a.e. in }\Omega\}$ instead of $\varphi\in\Phi_{ad}$. Thus, the integral volume constraint is replaced by an equality constraint, which means that we prescribe the exact volume fraction in advance. As already discussed in Remark \ref{r:EqualityConstraint}, this does not change the analytical results presented in the previous section. The only difference is, that the Lagrange multipliers $\lambda_\epsilon\in\R$ are also allowed to be negative and the complementarity conditions are fulfilled trivially.

\subsection{Description of algorithm}On the reduced problem formulation we apply an extension of the projected gradient method without requiring the existence of a gradient. In addition we allow for a variable scalar scaling $\zeta_k>0$ of the derivative,  as well as the use of a variable metric, which can include second order information. The step length is determined by Armijo backtracking. For more details, see \cite{Buchkapitel, RupprechtBlank}.

\begin{algo}\label{algorithm1}
\quad 
\begin{algorithmic}[1]
\STATE Choose $0<\beta < 1$, $0<\sigma<1$ and $\bfphi_0 \in\widetilde\Phi_{ad}$; set $k := 0$.
\WHILE{$k \leq k_{\textrm{max}}$}
	\STATE Choose $\zeta_k>0$ and an inner product $a_k$.
	\STATE Calculate the minimum $\overline\varphi_k=\mathcal P_k(\bfphi_k)$ of the projection-type subproblem
\begin{align}
	\min_{y  \in\widetilde{\Phi}_{ad}} \quad&\frac{1}{2}\| y -\bfphi_k\| _{a_k}^2 
+ \zeta_k j_\varepsilon'(\bfphi_k)(y -\bfphi_k).
\label{eq:projproblvm}
\end{align}
	\STATE Set the search direction $\bfv_k := \overline\varphi_k - \bfphi_k$
	\IF{$\sqrt{\varepsilon\gamma}\|\nabla\bfv_k\|_{L^2} < \textrm{tol}$ }\label{a:StopCrit}
		\RETURN
	\ENDIF	
	\STATE Determine the step length $\alpha_k:= \beta^{m_k}$ with minimal $m_k\in\N_0$ such that\\ $j_\varepsilon(\bfphi_k + \alpha_k v_k) \leq j_\varepsilon(\bfphi_k) + \alpha_k \sigma \spr{j_\varepsilon'(\bfphi_k),v_k}.$
	\STATE Update $\bfphi_{k+1} := \bfphi_k + \alpha_k\bfv_k$,
	\STATE $k:=k+1$,
\ENDWHILE
\end{algorithmic}
\end{algo}

The maximal number of iterations $k_{\textrm{max}}$ is set to $10^{5}$ in the experiments below. The directional derivative $j'_\epsilon(\varphi_k)v$ is given in $\eqref{e:ReducedObjFctlDirDer}$. We start with small $\zeta_k$ to enhance the convergence of the PDAS method and increase it slowly in every step until $\zeta_k=1$. As an inner product $a_k$ we start choosing $a_k(\varphi_1,\varphi_2) = \varepsilon\gamma\int_\Omega \nabla \varphi_1\cdot\nabla \varphi_2\dx$. To get faster convergence, we update the inner product $a_k$ in every step by a BFGS update whenever possible. Note that the solution $\overline\varphi_k$ to the projection-type subproblem \eqref{eq:projproblvm} is formally given by $\overline\varphi_k = P_{a_k}(\bfphi_k-\zeta_k \nabla_{a_k}j_\varepsilon(\bfphi_k))$, where $P_{a_k}$ denotes the orthogonal projection onto $\widetilde{\Phi}_{ad}$ with respect to the inner product $a_k$ and $\nabla_{a_k}j_\varepsilon(\bfphi_k)$ denotes the Riesz representative of $j_\varepsilon'(\bfphi_k)$ with respect to the inner product $a_k$. This is only formally since $j_\varepsilon$ need not be differentiable with respect to the norm induced by $a_k$. The subproblem \eqref{eq:projproblvm} is solved by a primal dual active set (PDAS) method, see \cite{bgsssv2010, ItoKunischPDAS}. 

\begin{remark}
In \cite{RupprechtBlank} it is shown that every accumulation point of the sequence $(\varphi_k)_{k\in\N}$ in the $H^1(\Omega)\cap L^\infty(\Omega)$ topology is a stationary point of $j_\varepsilon$ and that \linebreak[4]$\lim_{k\to\infty}\|v_k\|_{H^1(\Omega)}=0$. Additionally mesh independence can be observed.
\end{remark}

For all of our numerical experiments we consider the compliance problem, see Example~\ref{ex:MeanComp}, and assume to have only the external surface load $\b g$ as well as no eigenstrain, hence $\b f\equiv\b 0$ and $\overline\Epsilon\equiv0$. Thus we choose $h_\Omega(x,\b u)=0$, $h_\Gamma(x,\b u)=\b g\cdot\b u$. In the following examples $\b g$ will always be an element in $\b L^\infty(\Gamma_g)$ and so Assumption \ref{a:ElastObjectiveFctl} is fulfilled. Thus the objective functional is given in the following numerical experiments by

$$j_\epsilon(\varphi)=\int_{\Gamma_g}\b g\cdot\b S(\varphi)\ds+\gamma\int_\Omega\frac\epsilon2\left|\nabla\varphi\right|^2+\frac{1}{2\epsilon}\left(1-\varphi^2\right)\dx.$$
The used elasticity tensor interpolates between its two values quadratically. Inside the first material (represented by $\varphi=1$) and the second material (represented by $\varphi=-1$) it is given by the two Lam\'e constants $\lambda_1,\mu_1$ and $\lambda_2,\mu_2$, respectively. We choose $\C(\varphi)\Epsilon:=2\iota_{\delta_\mu}(\varphi)\mu_2\Epsilon+\iota_{\delta_\lambda}(\varphi)\lambda_2\mathrm{tr}(\Epsilon)I$ where $\iota_\delta(\varphi)=0.25\left(1-\delta\right)\varphi^2-0.5\left(1-\delta\right)\varphi+0.25(1-\delta)+\delta$, thus $\iota_\delta(1)=1$ and $\iota_\delta(-1)=\delta$, and $\delta_{\mu}:=\nicefrac{\mu_1}{\mu_2}$, $\delta_\lambda:=\nicefrac{\lambda_1}{\lambda_2}$.\\

The phase field $\varphi$ and the state equation are discretized using standard piecewise linear finite elements. An adaptive mesh is implemented, which is fine on the interface and coarse in the bulk region as described in \cite{BNS04}. All appearing integrals are computed by exact quadrature rules.

\subsection{Optimal design of a cantilever beam}\label{s:Cantilever}
The first experiment is carried out in the design domain $\Omega=(-1,1)\times(0,1)$ and we choose $\Gamma_D:=\{(x,y)\in\partial\Omega\mid x=-1\}$ and the support of the force $\b g$ will be concentrated on $\Gamma_g^0:=\{(x,y)\in\partial\Omega\mid x\geq 0.75, y = 0\}\subset\Gamma_g$. The surface load is chosen to be $\b g=\left(0,-250\right)^T\chi_{\Gamma_g^0}$. The configuration is sketched in Figure \ref{fig:Config1} and the initial shape always corresponds to $\varphi\equiv 0$. Thus, also $\beta=0$ is chosen for the volume constraint. The first material, represented by $\varphi=1$ is given by the constants $\lambda_1=\mu_1=5000$ and the second material (i.e. $\varphi=-1$) is a more elastic material with $\lambda_2=\mu_2=10$. Lam\'e constants with such a pronounced contrast also appear in applications, compare \cite{sale2013structural}. A uniform mesh size of $h=2^{-6}$ is chosen on the bulk and the interfacial layer is refined such that there are 8 mesh points across the interface. Since the interface thickness is proportional to $\varepsilon$, the mesh has to be chosen finer on the interface as $\varepsilon$ gets smaller.

\begin{figure}
  \centering
		\subfigure[Configuration of the cantilever beam\label{fig:Config1}]{%
{\epsfig{file=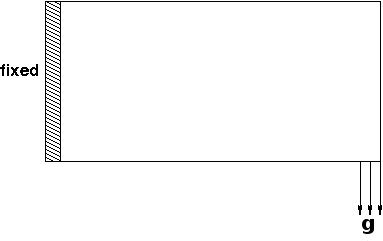,width=0.48\textwidth}}%
}\quad
\subfigure[Optimal configuration for $\gamma=0.002$. The stiff material is blue.\label{fig:Config2}]{%
{\epsfig{file=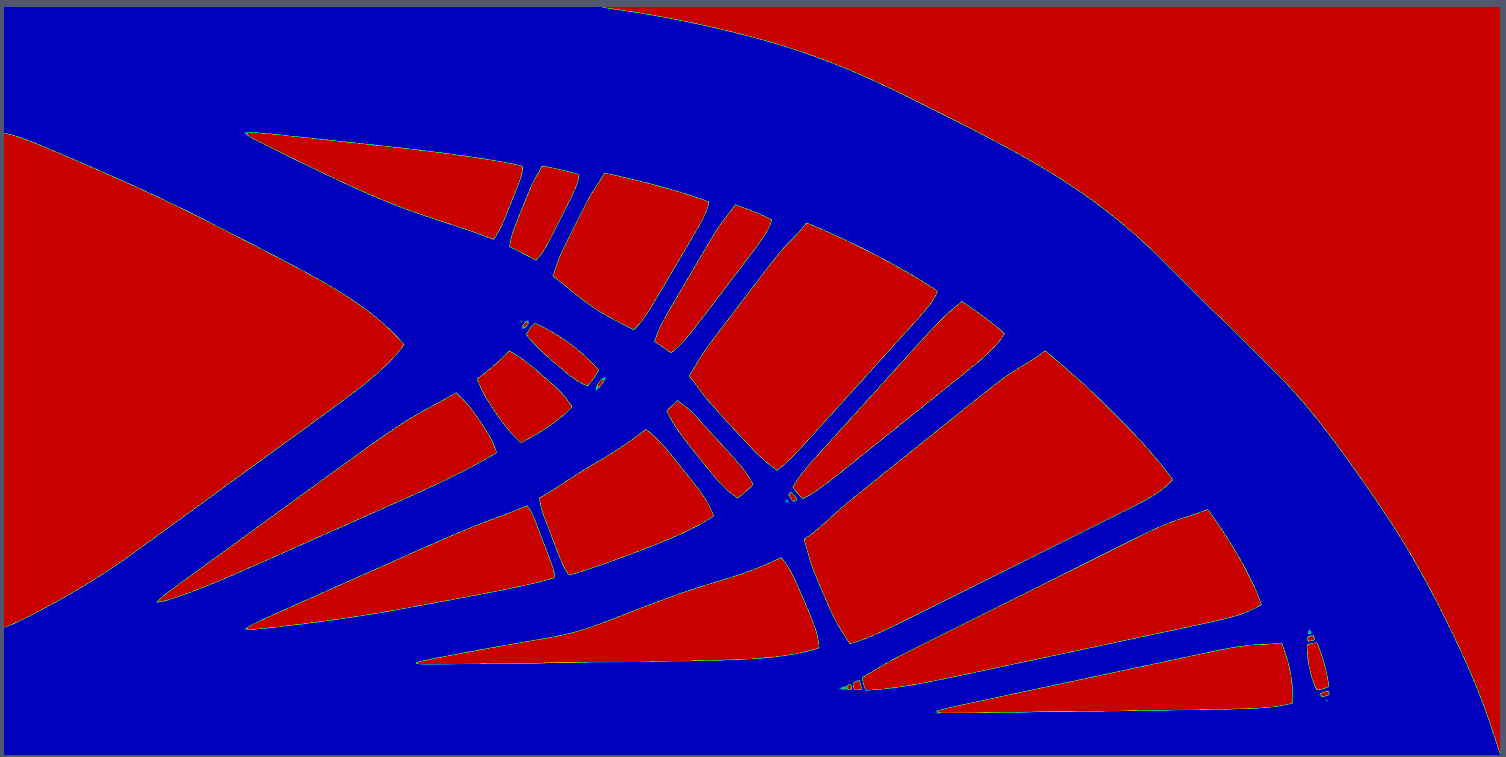,width=0.48\textwidth}}%
}  %\epsfig{file=clbeamneu1,width=0.5\textwidth}\hfill
	\caption{Cantilever beam}
 \label{fig:Config}
\end{figure}

\begin{figure}
  \centering
	\subfigure[$\epsilon=0.06$]{%
{\epsfig{file=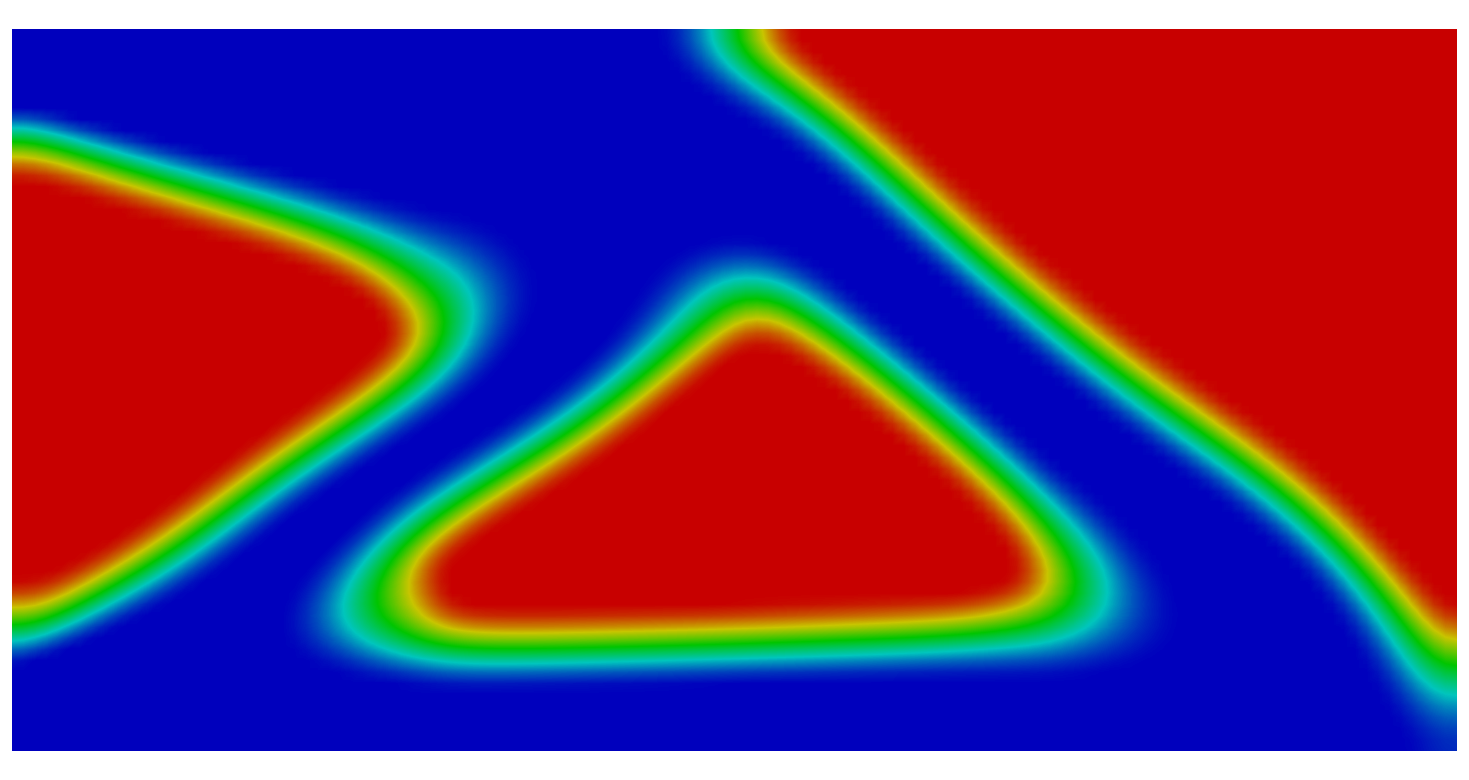,width=0.3\textwidth}}%
}\quad
\subfigure[$\epsilon=0.015$]{%
{\epsfig{file=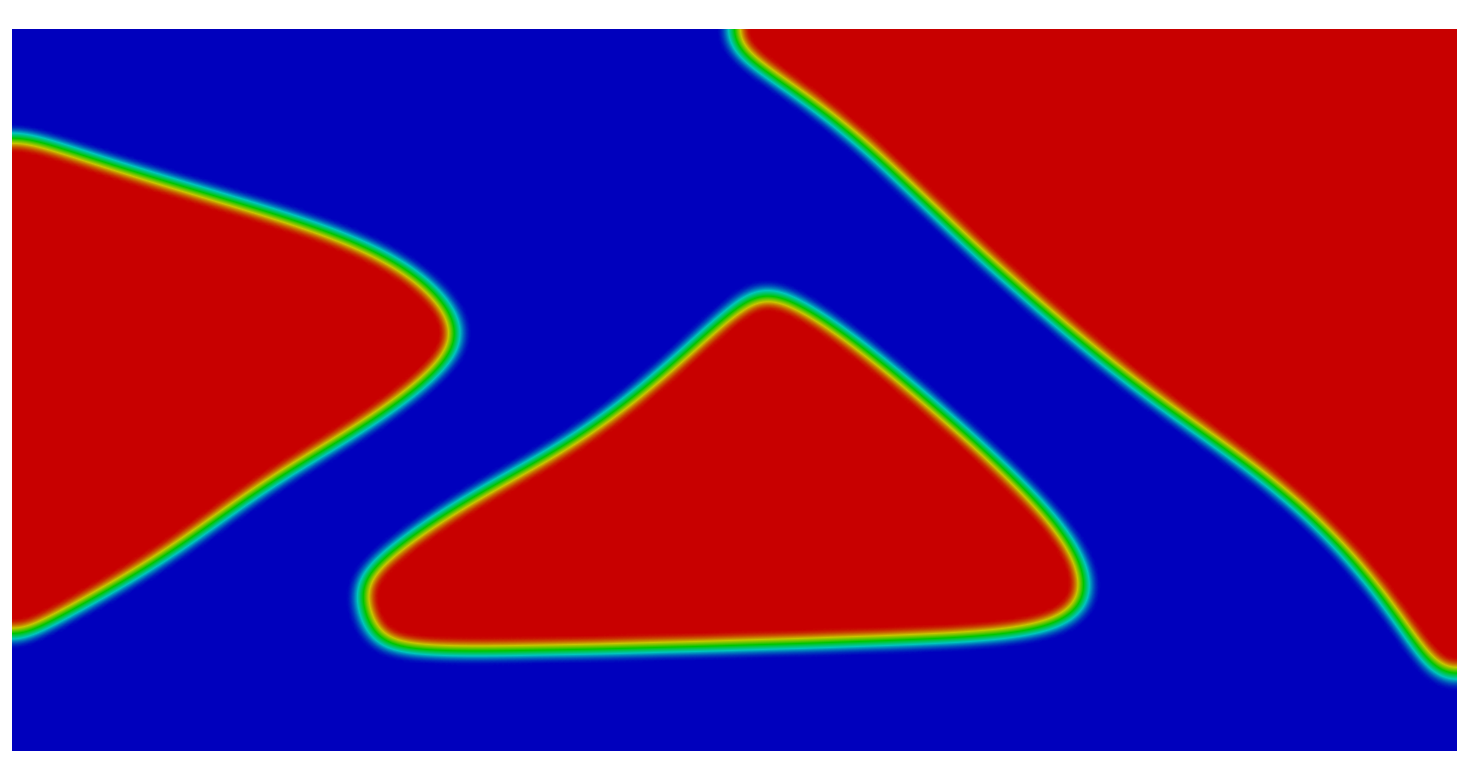,width=0.3\textwidth}}}\quad
\subfigure[$\epsilon=0.001875$]{%
{\epsfig{file=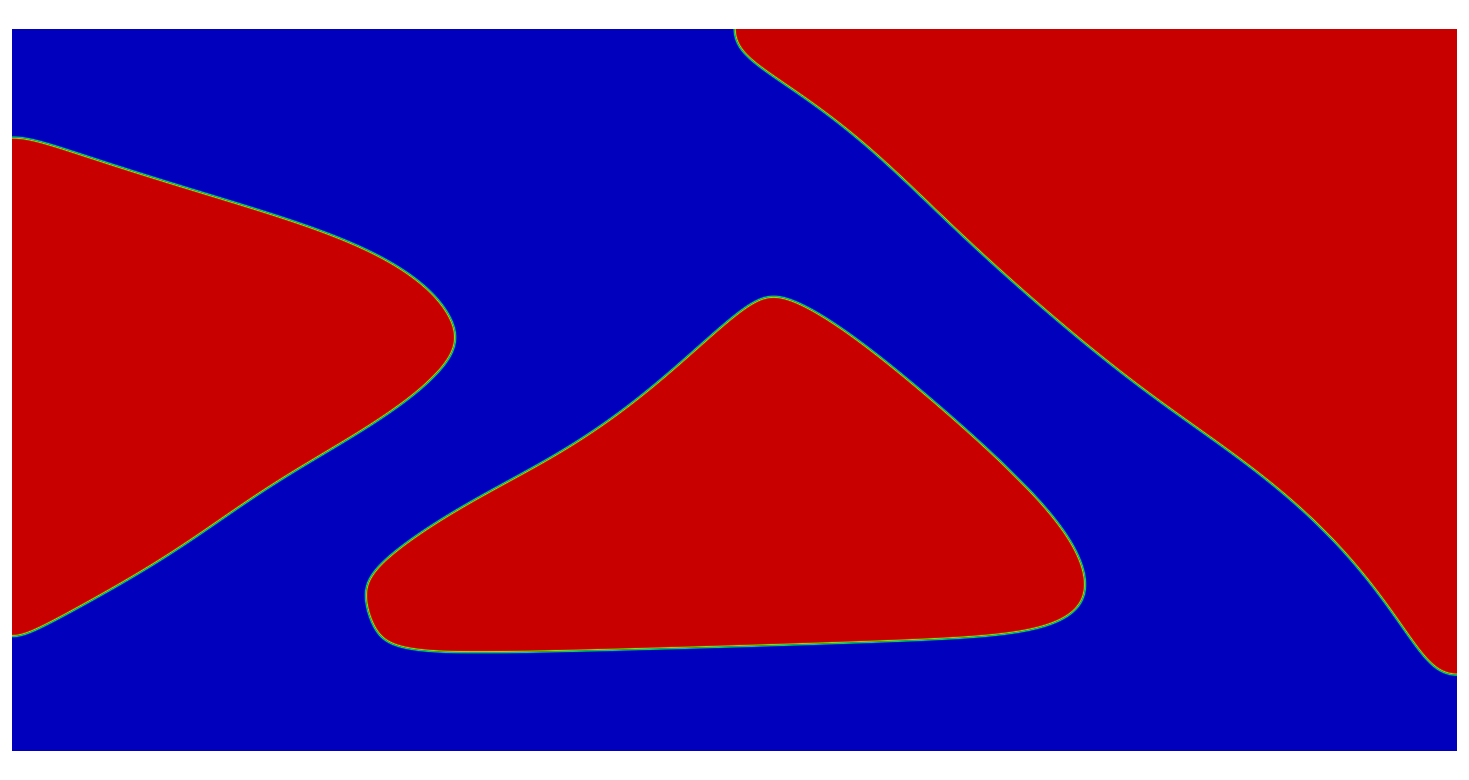,width=0.3\textwidth}}}%
	\caption{Optimal material configuration $\varphi_\epsilon$ for the cantilever beam for $\gamma=0.5$ and different values of $\epsilon$ (stiff material in blue, weak material in red and interface in green).}
  \label{fig:DiffEps}
\end{figure}

\begin{figure}
\def\tabularxcolumn#1{m{#1}}
\begin{tabularx}{\linewidth}{@{}cX@{}}
\subfigure[The solid blue line shows $\|\varphi_\epsilon-\tilde\varphi_0\|_{L^1(\Omega)}$. The approximated diffuse interface error is depicted in the red, dashed line.\label{fig:ErrorEps1}]{%
{\epsfig{file=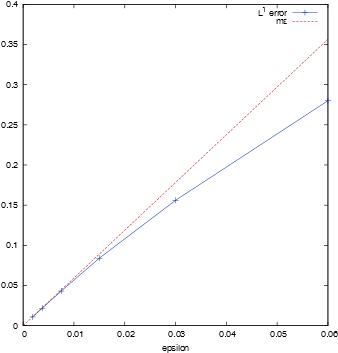,width=0.46\textwidth}}%
}
&\vspace*{-4cm}
\begin{tabular}{c}
\subfigure[The cost functional $j_\epsilon(\varphi_\epsilon)$ seems to converge as $\epsilon\searrow0$.\label{fig:ErrorEps2}]{%
{\epsfig{file=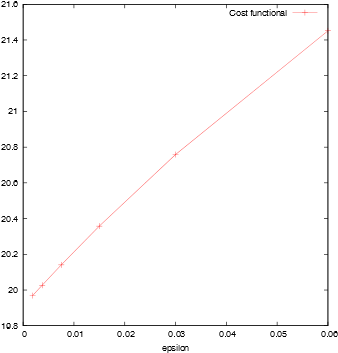,width=0.46\textwidth,height=100pt}} %
}\\
\subfigure[The Lagrange multiplier $\lambda_\epsilon$ seems to converge as $\epsilon\searrow0$.\label{fig:ErrorEps3}]{%
{\epsfig{file=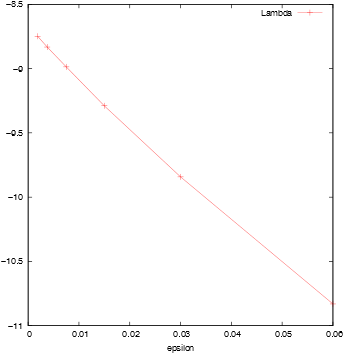,width=0.46\textwidth,height=100pt}} %
}
\end{tabular}
\end{tabularx}
\caption{Data corresponding to Figure \ref{fig:DiffEps}.}
\end{figure}

%
%
%
%\begin{figure}
%\centering
	%\subfigure[The blue line shows the error $\|\varphi_\epsilon-\left(2\chi_{\{\varphi_{\epsilon_{\min}}>0\}}-1\right)\|_{L^1(\Omega)}$. The error expected only from the diffuse interface is depicted in green.\label{fig:ErrorEps1}]{%
%{\epsfig{file=L1error,width=0.46\textwidth}}%
%} \quad
	%\subfigure[The cost functional $j_\epsilon(\varphi_\epsilon)$ seems to converge as $\epsilon\searrow0$.\label{fig:ErrorEps2}]{%
%{\epsfig{file=Cost,width=0.46\textwidth}} %
%}	
%\caption{Convergence results for $\epsilon\searrow0$}
%%\caption{The $L^1$ error plots $\|\varphi_\epsilon-\left(2\chi_{\{\varphi_{\epsilon_{\min}}>0\}}-1\right)\|_{L^1(\Omega)}$ and compares it to the error expected only from the diffuse interface (green line) (left). The cost functinoal $j_\epsilon(\varphi_\epsilon)$ seems to converge as $\epsilon\searrow0$ (right).}
  %\label{fig:ErrorEps}
%\end{figure}

Results for minimizing the mean compliance of a cantilever beam for three different phase field parameters $\epsilon$ are shown in Figure \ref{fig:DiffEps}. Here we chose the weighting factor for the Ginzburg Landau energy $\gamma=0.5$. Already for $\epsilon=0.06$ one obtains the same structure as for the smallest value of the phase field parameter, i.e. the right qualitative behaviour of the sharp interface minimizer. We also remark that for the smallest value of $\epsilon$, the interface is already very thin and almost not visible anymore.\\
We want to study the convergence of the minimizers $(\varphi_\epsilon)_\epsilon$. For this purpose we denote by $\varphi_0$ a minimizer for the sharp interface problem. As this minimizer is a priori unknown, we approximate $\varphi_0$ by $\tilde\varphi_0:=2\chi_{\{\varphi_{\epsilon_{\min}}>0\}}-1$, where $\epsilon_{\min}:=0.001875$. Thus the optimal interface $\Gamma_0$ is approximated by the zero level set of $\varphi_{\epsilon_{\min}}$, denoted by $\tilde\Gamma_0$. The difference $\|\varphi_\epsilon-\tilde\varphi_0\|_{L^1(\Omega)}$ can be separated into two terms,
\begin{align}\label{e:NumericErrorDecompose}\|\varphi_\epsilon-\tilde\varphi_0\|_{L^1(\Omega)}\sim\|\varphi_\epsilon-\varphi_\epsilon^o\|_{L^1(\Omega)}+\|\varphi_\epsilon^o-\tilde\varphi_0\|_{L^1(\Omega)}\end{align}
where $(\varphi_\epsilon^o)_\epsilon$ denotes the constructed recovery sequence corresponding to $\tilde\varphi_0$ and exposes a $\sin$-profile normal to $\tilde\Gamma_0$, see \cite{blowey_elliot}. We want to determine which term of the error in $\eqref{e:NumericErrorDecompose}$ is the dominating part. We notice that the first term on the right-hand side of $\eqref{e:NumericErrorDecompose}$ can be considered as the distance of the zero level sets of $\varphi_\epsilon$ and $\varphi_{\epsilon_{\min}}$ and the second term describes the error resulting from the diffuse interface profile. The 1D error of the $\sin$-profile compared to a characteristic function is given by $\int_{-\nicefrac{\epsilon\pi}{2}}^{\nicefrac{\epsilon\pi}{2}}|\sin(\frac x\epsilon)-\mathrm{sgn}(x)|\dx=(\pi-2)\epsilon$. Thus, the $L^1(\Omega)$-error in our $2D$ setting can be approximated by $\|\varphi_\epsilon^o-\tilde\varphi_0\|_{L^1(\Omega)}\approx P_\Omega(\{\varphi_0=1\})(\pi-2)\epsilon$. As the perimeter $P_\Omega(\{\varphi_0=1\})$ is not known, we extrapolate the given sequence $(E_\epsilon(\varphi_\epsilon))_\epsilon$ numerically to $\epsilon=0$, which gives $\lim_{\epsilon\searrow0} E_\epsilon(\varphi_\epsilon)\approx e_0:=8.1754$. As mentioned above, from the $\Gamma$-convergence result of \cite{blowey_elliot, modica} it is expected that $e_0\approx \frac\pi2 P_\Omega(\{\varphi_0=1\})$. Hence, we may approximate $\|\varphi_\epsilon^o-\tilde\varphi_0\|_{L^1(\Omega)}\approx m\epsilon$ with $m:=\frac{2(\pi-2)}{\pi}e_0$.\\
In Figure \ref{fig:ErrorEps1} we depict now the difference $\|\varphi_\epsilon-\tilde\varphi_0\|_{L^1(\Omega)}$ (solid blue line) together with the approximated diffuse interface error $m\epsilon$ (dashed red line). We see that for $\epsilon<0.02$ the $L^1$-difference of the minimizer $\varphi_\epsilon$ and the approximated minimizer $\tilde\varphi_0$ becomes tangential to the line $m\epsilon$. This indicates that the error resulting from the diffuse interface profile dominates the total approximated error in $\eqref{e:NumericErrorDecompose}$ and that the distance of the zero level sets becomes comparably small. Hence, the level set of $\varphi_\epsilon$ is already for $\epsilon<0.02$ a good approximation of the optimal interface $\Gamma_0$.\\

To complete this picture we also give a plot of the minimal functional values $j_\epsilon(\varphi_\epsilon)$ and the Lagrange multipliers $\lambda_\epsilon$ for the calculated $\epsilon$ values in Figure \ref{fig:ErrorEps2} and \ref{fig:ErrorEps3}. One sees that $(j_\epsilon(\varphi_\epsilon))_\epsilon$ is monotonically decreasing as $\epsilon$ decreases and seems to converge linearly to a specific value, supposedly a minimal value for $j_0$, compare also Theorem \ref{c:ConvMinElastConv}. Likewise, $(\lambda_\epsilon)_\epsilon$ converges linearly to a limit value $\lambda_0$, compare Theorem \ref{t:ConvOPtSys}.\\
To study the influence of the perimeter penalization, we carried out the same calculations for a smaller weighting factor $\gamma$. We can control the appearance of fine structures by the choice of $\gamma$. As an example, we refer to Figure \ref{fig:Config2}, where we used the parameters $\gamma=0.002$, $\epsilon=0.001$. This verifies numerically that the regularization yields a well-posed problem but still gives desired optimal structures. Moreover, the influence of regularization parameters on the fineness of the structure is in accordance to other methods (see e.g. \cite{bendsoe2003topology}). For the computation we chose the inner product
\begin{align*}
	a_k(\varphi_1, \varphi_2) := \gamma\varepsilon \int_\Omega\nabla \varphi_1\cdot\nabla \varphi_2\dx - 2 \int_\Omega \C'(\varphi_k)(\varphi_1) \Epsilon(\b z):\Epsilon(\b u_k) \dx,
\end{align*}
which depends on the current iterate $\varphi_k$ and which includes second order information of the Ginzburg-Landau energy, as well as of the compliance part. Here, we used $\b u_k:=\b S(\varphi_k)$ and $\b z:=\b S'(\varphi_k)\varphi_2\in \b H^1_D\left(\Omega\right)$ which is given as
the solution of the linearized state equation
\begin{align*}
	\int_\Omega \C(\varphi_k) \Epsilon(\b z):\Epsilon(\b v)\dx = -\int_\Omega \C'(\varphi_k)\varphi_2\Epsilon(\b u_k):\Epsilon(\b v)\dx + \int_{\Gamma_g}\b g\cdot \b v\ds\quad \forall\b v\in \b H^1_D\left(\Omega\right).
\end{align*}

\subsection{Optimal material distribution within a wing}\label{s:Wing}
We now consider a different geometry for the overall container $\Omega$ in a three dimensional setting. This example shows that we can also use different geometries, i.e. different choices of $\Omega$, and work in a three dimensional setting.\\
We perform the same optimization strategy as above and optimize the material configuration within a wing of an airplane. This example is to be considered as an outlook on possible applications. One wants to have a composite material in order to obtain a high ratio of stiffness to weight. Hence we use one very stiff material ($\lambda_1=5000, \mu_1=5000$) and a material representing the light material ($\lambda_2=100$, $\mu_2=100$). As geometry we use a three dimensional NACA 0018 airfoil configuration with three holes in it. The configuration can be seen in Figure \ref{fig:Wing}.\\
The considerations of the previous example have shown that we do not have to choose $\epsilon$ too small in order to obtain the right qualitative behaviour and so we use here $\epsilon=0.06$. Moreover, the weighting factor for the Ginzburg-Landau regularization is chosen quite small, i.e. $\gamma=10^{-4}$. The boundary force $\b g(x,y,z)=\left(0,0.03\sqrt{1-\left(0.5y\right)^2},0\right)$ is of elliptic form, which is typical for the lift force acting on an airplane wing, compare for instance \cite{dorand1922influence}. Its support is on $\Gamma_g^0:=\{(x,y,z)\in\partial\Omega\mid y>0\}$. The Dirichlet boundary $\Gamma_D=\{(x,y,z)\in\partial\Omega\mid z=0\}$ is the part where the wing is attached to the airplane (left-hand side in Figure \ref{fig:Wing}). The optimized material configuration is shown in Figure \ref{fig:Wing}, where the blue material is the stiff material. We also give a picture of the weak material, i.e. the set $\{\varphi_\epsilon<0\}$, see Figure \ref{fig:Wing}, in order to see the hypersurface separating the two materials, together with various cross sections of the wing.

\begin{figure}
\centering
	\subfigure{%
{\epsfig{file=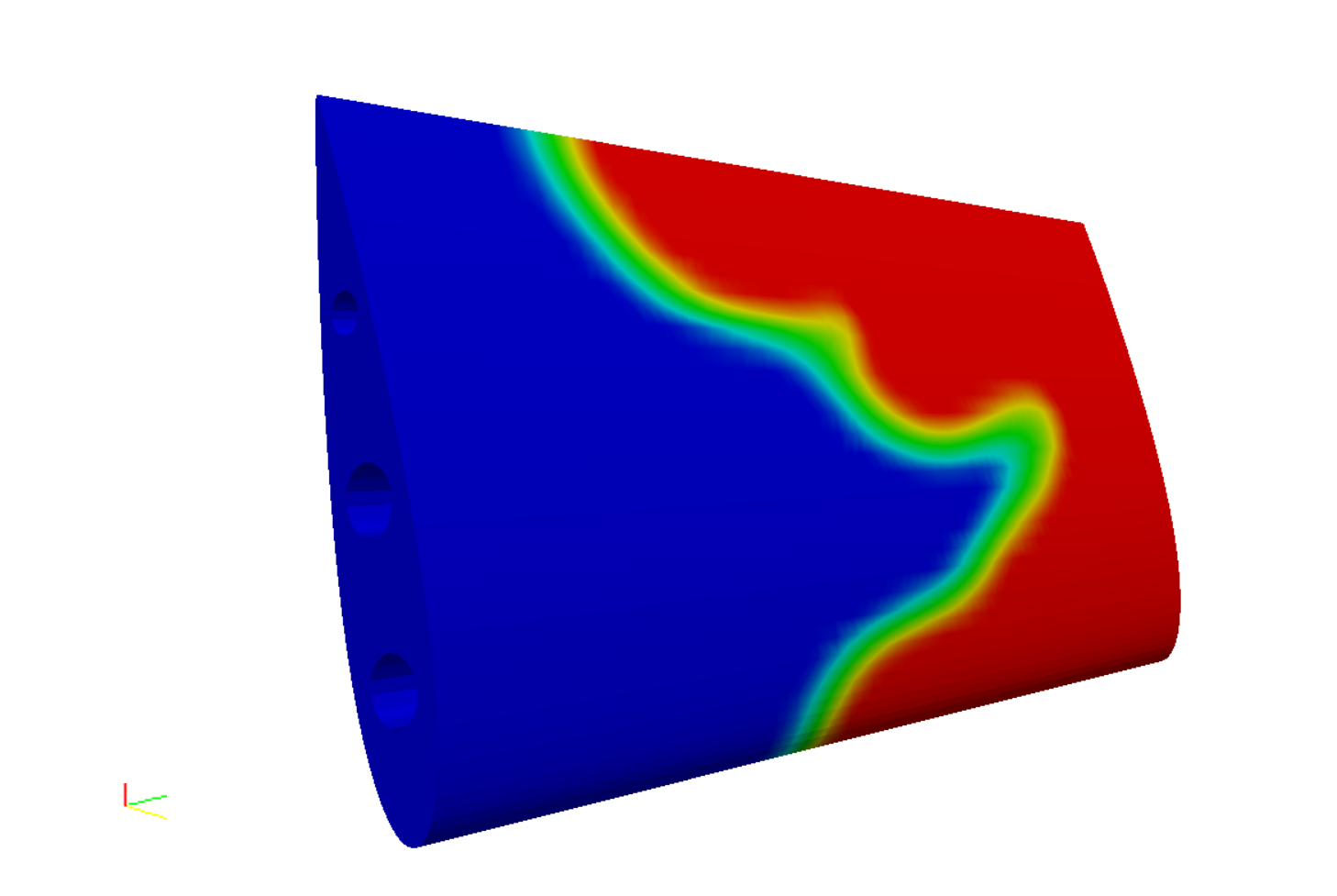,width=0.42\textwidth}}%
\label{fig:Wing1}} 
	\subfigure{%
{\epsfig{file=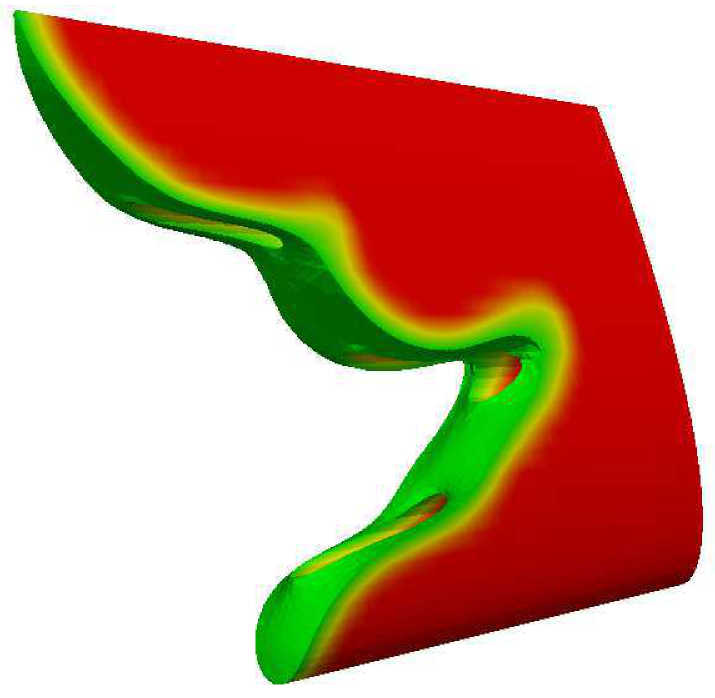,width=0.19\textwidth}} \label{fig:Wing2}%
}	\subfigure{%
{\epsfig{file=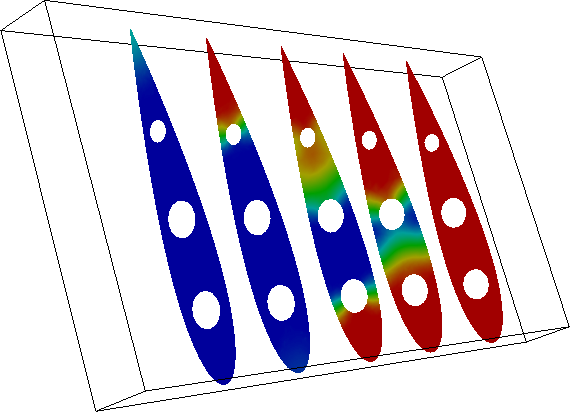,width=0.3\textwidth}} \label{fig:Wing3}%
}	
\caption{Optimal material configuration of a NACA 0018 airfoil wing with 3 holes where blue represents the stiff material.}
  \label{fig:Wing}
\end{figure}

\subsection{Influence of eigenstrain}
\begin{figure}
\centering
	\subfigure{%
{\epsfig{file=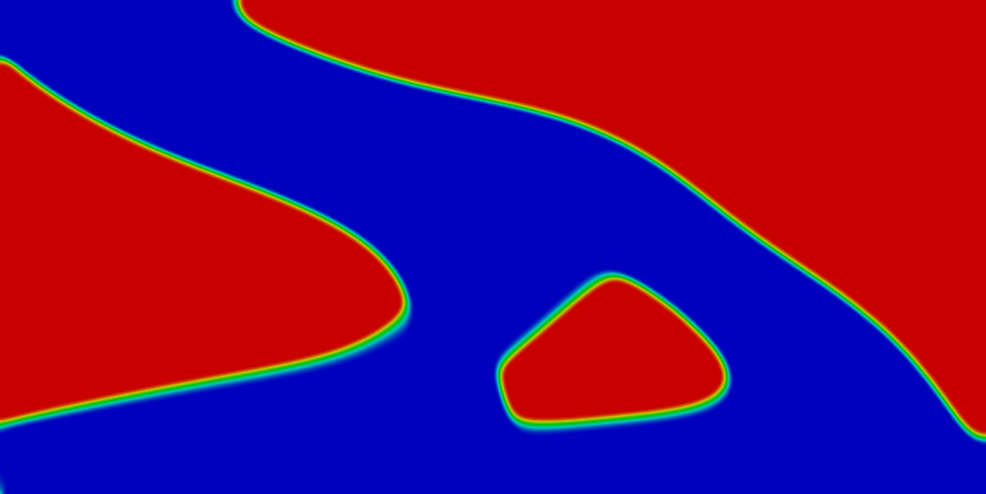,width=0.42\textwidth}}%
}\hspace{1cm}
	\subfigure{%
{\epsfig{file=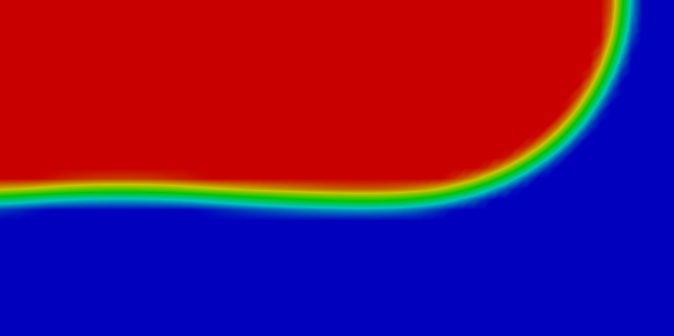,width=0.42\textwidth}}
}	
\caption{Optimal designs including eigenstrain.}\label{fig:Eigenstrain}
\end{figure}

Left hand side of Figure \ref{fig:Eigenstrain}: We take $\overline\Epsilon(\varphi) = \delta\varphi I$ with $\delta = 0.08$, $\varepsilon=0.01$ and $\gamma=0.5$ and $\lambda_1=\mu_1=5000$ and $\lambda_2=\mu_2=5$ (strong and weak material).
Right hand side of Figure \ref{fig:Eigenstrain}: We take $\overline\Epsilon(\varphi) = \delta\varphi \begin{pmatrix}-1&0\\0&1\end{pmatrix}$ with $\delta = 0.01$, $\varepsilon=0.04$ and $\gamma=0.5$ and $\lambda_1=\mu_1=\lambda_2=\mu_2=5000$ (homogeneous material). The rest of the parameters are the same as in Section \ref{s:Cantilever}.

\section{Derivation of the optimality conditions}\label{s:ProofOptcond}
In this section we will give the proofs of Theorem \ref{t:ElastGeomVarSharp} and Theorem \ref{t:PhseOptCond}. We start with showing the differentiability of $t\mapsto\left(\b S\left(\varphi\circ T_t^{-1}\right)\circ T_t\right)$ at $t=0$ for $T\in\mathcal T_{ad}$ and $\varphi\in L^1(\Omega)$, $|\varphi|\leq1$ and deriving the validity of $\eqref{e:ElastPhaseDotUEpsilonEquationWeak}$:

\begin{lemma}\label{l:DerOptCondLem}Assume \ref{a:ElastOmega}-\ref{a:ElastDiffAss}.	Let $\varphi\in L^1(\Omega)$ with $|\varphi|\leq1$ a.e. in $\Omega$ and $T\in\mathcal T_{ad}$ chosen. We define $\varphi(t):=\varphi\circ T_t^{-1}$ and $\b u(t):=\b S(\varphi(t))$ for $|t|\ll1$. Then $t\mapsto(\b u(t)\circ T_t)\in\b H^1_D(\Omega)$ is differentiable at $t=0$ and $\dot{\b u}[V]:=\partial_t|_{t=0}\left(\b u(t)\circ T_t\right)\in\b H^1_D(\Omega)$ is the unique solution of $\eqref{e:ElastPhaseDotUEpsilonEquationWeak}$ with $\varphi_0$ replaced by $\varphi$ and $\b u_0$ replaced by $\b u:=\b S(\varphi)$.
\end{lemma}
\begin{proof}
The idea is to apply the implicit function theorem and hence we define for $\tau_0>0$ small enough the function $F:(-\tau_0,\tau_0)\times\b H^1_D(\Omega)\to\left(\b H^1_D\left(\Omega\right)\right)'$ by
	\begin{align*}F(t,\b u)(\b v)&=\int_\Omega\C(\varphi)\frac12\left(\nabla T_t^{-1}\nabla\b u+\der\b u\der T_t^{-1}\right):\frac12(\nabla T_t^{-1}\nabla\b v+\der\b v\der T_t^{-1})\det\der T_t\dx-\\
&-\int_\Omega\C\left(\varphi\right)\overline\Epsilon\left(\varphi\right):\frac12\left(\nabla T_t^{-1}\nabla\b v+\der\b v\der T_t^{-1}\right)\det\der T_t\dx-\\
&-\int_\Omega\b f\cdot\left(\b v\circ T_t^{-1}\right)\dx-\int_{\Gamma_g}\b g\cdot\left(\b v\circ T_t^{-1}\right)\ds.\end{align*}
	Using the calculation rules $\nabla\left(\b v\circ T_t\right)=\nabla T_t\left(\nabla\b v\right)\circ T_t$, $\der\left(\b v\circ T_t\right)=\left(\der\b v\right)\circ T_t\der T_t$ for $\b v\in\b H^1(\Omega)$	we can establish

\begin{align*}
	&F(t,\b u(t)\circ T_t)(\b v)=\int_\Omega\C(\varphi(t))\left(\Epsilon(\b u(t))-\overline\Epsilon\left(\varphi(t)\right)\right):\Epsilon(\b v\circ T_t^{-1})\dx-\\
	&-\int_\Omega\b f\cdot\left(\b v\circ T_t^{-1}\right)\dx-\int_{\Gamma_g}\b g\cdot\left(\b v\circ T_t^{-1}\right)\ds=0
\end{align*}
where we made use of $\b v\circ T_t^{-1}\in\b H^1_D(\Omega)$ if $\b v\in\b H^1_D(\Omega)$ by the particular choice of $T\in{\mathcal T}_{ad}$. Besides, $\der_uF\left(0,\b u\right):\b H^1_D(\Omega)\to\left(\b H^1_D(\Omega)\right)'$, given by

$$\der_uF\left(0,\b u\right)\left(\b u\right)\left(\b v\right)=\int_\Omega\C\left(\varphi\right)\Epsilon\left(\b u\right):\Epsilon\left(\b v\right)\dx\quad\forall \b u,\b v\in\b H^1_D(\Omega)$$
is by Lax-Milgram's theorem an isomorphism. And so we can apply the implicit function theorem to obtain differentiability of $(-\tau_0,\tau_0)\ni t\mapsto\left(\b u(t)\circ T_t\right)\in\b H^1_D(\Omega)$ at $t=0$ together with $\dot{\b u}\left[V\right]:=\partial_t|_{t=0}\left(\b u(t)\circ T_t\right)$, $\der_uF(0,\b u)\dot{\b u}[V]=-\partial_tF(0,\b u)$ and obtain therefrom $\eqref{e:ElastPhaseDotUEpsilonEquationWeak}$.
\end{proof}

Now we can directly proof the validity of the optimality system for the sharp interface problem:\\

{\em Proof of Theorem \ref{t:ElastGeomVarSharp}:}
The formula for the first variation of the perimeter functional can for instance be found in \cite[10.2]{giusti}. The volume integrals appearing in the objective functional can be differentiated directly by using change of variables. To handle the boundary integrals, we use the calculation rules derived in \cite[Chapter 9, Section 4.2]{delfour2} to see
\begin{align*}\partial_t|_{t=0}\int_{T_t(\Gamma_g)} h_\Gamma\left(s,\b u_0(t)\right)\ds&= \partial_t|_{t=0}\int_{\Gamma_g} h_\Gamma\left(T_t(s),\b u_0(t)\circ T_t\right)\omega_t\ds
\end{align*}
where $\omega_t=\left|\det\der T_t\der T_t^{-T}\b n\right|$, $\b u_0(t):=\b S(\varphi_0\circ T_t^{-1})$. The derivative of $\omega_t$ with respect to $t$ at can be calculated by $\partial_t|_{t=0}\omega_t=\div V(0)-\b n\cdot\nabla V(0)\b n.$ For more details we refer to \cite{delfour2}. And so we arrive in
\begin{align*}\partial_t|_{t=0}\int_{T_t(\Gamma_g)} h_\Gamma\left(s,\b u_0(t)\right)\ds&= \int_{\Gamma_g} \der h_\Gamma\left(s,\b u_0\right)\left(V(0),\dot{\b u}_0\left[V\right]\right)+\\
&+h_\Gamma\left(s,\b u_0\right)\left(\div V(0)-\b n\cdot\nabla V(0)\b n\right)\ds
\end{align*}
where $\dot{\b u}_0[V]:=\partial_t|_{t=0}\left(\b u_0(t)\circ T_t\right)$ is already determined by Lemma \ref{l:DerOptCondLem}. \\

 The existence of a Lagrange multiplier for the integral constraint follows as in \cite[Lemma 7.5]{hecht}, see also \cite{GarckeHechtStokes}. For the sake of readability, we restate here the main steps of this proof. First we may assume without loss of generality that $\int_\Omega\varphi_0\dx=\beta|\Omega|$, otherwise any transformation $T\in\mathcal T_{ad}$ will yield admissible comparison functions $\varphi_0\circ T_t^{-1}\in\Phi_{ad}^0$ for $|t|\ll1$ and in this case $\lambda_0=0$ is the desired Lagrange multiplier. Considering the case $\int_\Omega\varphi_0\dx<\beta|\Omega|$, we choose some $W\in\mathcal V_{ad}$ with associated transformation $S\in\mathcal T_{ad}$ such that $\int_\Omega\varphi_0\div W(0)\dx=-1$ and define $g:[-t_0,t_0]\times[-s_0,s_0]\to\R$ by $g(t,s):=-\int_\Omega\varphi_0\circ T_t^{-1}\circ S_s^{-1}\dx+\beta|\Omega|$ for $s_0, t_0>0$ small enough. Direct calculation yields $\partial_s|_{s=0}g(0,s)=-\int_\Omega\varphi_0\div W(0)\dx=1\neq 0$. And so we apply the implicit function theorem to obtain $s\in C^1((-\tau_0,\tau_0),\R)$ such that $g(t,s(t))=0$ for $|t|\ll1$, $s'(0)=-\partial_s|_{s=0} g(0,s)^{-1}\partial_t|_{t=0}g(t,0)=-\partial_t|_{t=0}g(t,0)$. Hence $\varphi_0\circ T_t^{-1}\circ S_{s(t)}^{-1}\in\Phi_{ad}^0$ for $|t|\ll1$ and thus $\partial_t|_{t=0}j_0\left(\varphi_0\circ (S_{s(t)}\circ T_t)^{-1}\right)=0$. One can then establish that 
\begin{align*}0&=\partial_t|_{t=0}j_0\left(\varphi_0\circ (S_{s(t)}\circ T_t)^{-1}\right)=\partial_s|_{s=0}j_0\left(\varphi_0\circ S_s^{-1}\right)s'(0)+\partial_t|_{t=0}j_0\left(\varphi_0\circ T_t^{-1}\right)=\\
&=\lambda_0\int_\Omega\varphi_0\div V(0)\dx+\partial_t|_{t=0}j_0\left(\varphi_0\circ T_t^{-1}\right)\end{align*}
where we defined $\lambda_0:=\partial_s|_{s=0}j_0\left(\varphi_0\circ S_s^{-1}\right)$. As we chose $\int_\Omega\varphi_0\div W(0)=-1<0$ we have that $\int_\Omega\varphi_0\circ S_s^{-1}\dx\leq\beta|\Omega|$, hence $\varphi_0\circ S_s^{-1}\in\Phi_{ad}^0$, for $0<s\ll1$. This shows $\lambda_0\geq 0$ and yields in particular that $\lambda_0$ is a Lagrange multiplier.
\qquad$\square$\\

Similarly, we directly establish the corresponding optimality system for the phase field problems:\\

{\em Proof of Theorem \ref{t:PhseOptCond}:}
We follow the lines of the proof of Theorem \ref{t:ElastGeomVarSharp}. The differential of the terms from the Ginzburg-Landau energy can be treated by direct calculation, compare for instance \cite[Lemma 7.5]{hecht}. 
\qquad$\square$

\section{Proof of the convergence results}\label{s:ProofConvRes} In this section we want to prove the convergence results stated in Theorem \ref{t:ElastConvGammaConv} and Theorem \ref{t:ConvOPtSys}. First, we want to give a proof of the $\Gamma$-convergence result of Theorem \ref{t:ElastConvGammaConv}. For this purpose, we start with the following lemma:

\begin{lemma}\label{l:ElastConvGammaConvFEcont}
	Under the assumptions \ref{a:ElastOmega}-\ref{a:ElastObjectiveFctl}, the function
	$$F_E:\left\{\varphi\in L^1(\Omega)\mid\left|\varphi\right|\leq1\text{ a.e. in }\Omega\right\}\ni\varphi\mapsto\int_\Omega h_\Omega\left(x,\b S(\varphi)\right)\dx+\int_{\Gamma_g}h_\Gamma\left(s,\b S\left(\varphi\right)\right)\ds$$
	is continuous in $L^1(\Omega)$. Besides we find, that $\b S:\{\varphi\in L^1(\Omega)\mid|\varphi|\leq 1\text{ a.e.}\}\to\b H^1_D(\Omega)$ is demicontinuous.
\end{lemma}

\begin{proof}
	Let $\left(\varphi_n\right)_{n\in\N}\subset L^1(\Omega)$ be a sequence such that $\left|\varphi_n\right|\leq1$ a.e. in $\Omega$ for every $n\in\N$ and $\lim_{n\to\infty}\left\|\varphi_n-\varphi\right\|_{L^1(\Omega)}=0$. In particular, this gives directly $\left|\varphi\right|\leq1$ a.e. in $\Omega$. Now let $\left(\varphi_{n_k}\right)_{k\in\N}$ be any subsequence of $\left(\varphi_n\right)_{n\in\N}$. Defining $\b u_{n_k}:=\b S(\varphi_{n_k})$ we see that it holds
	$$\int_\Omega\C\left(\varphi_{n_k}\right)\left(\Epsilon\left(\b u_{n_k}\right)-\overline\Epsilon\left(\varphi_{n_k}\right)\right):\Epsilon\left(\b u_{n_k}\right)\dx=\int_\Omega\b f\cdot\b u_{n_k}\dx+\int_{\Gamma_g}\b g\cdot\b u_{n_k}\ds\quad\forall k\in\N.$$
	Thus, by applying the inequalities of Korn, Young and H\"older and the uniform estimate on the elasticity tensor $\C$, see $\eqref{e:UniforEstimateElast}$, we obtain that
	$$\sup_{k\in\N}\left\|\b u_{n_k}\right\|_{\b H^1(\Omega)}<\infty.$$
	And so we find a subsequence $\left(\b u_{n_{k(l)}}\right)_{l\in\N}$ such that $\left(\b u_{n_{k(l)}}\right)_{l\in\N}$ converges weakly in $\b H^1(\Omega)$ to some $\b u\in\b H^1_D(\Omega)$ as $l\to\infty$. Using the uniform boundedness of the tensor-valued function $\C\in C^{1,1}\left(\left[-1,1\right],\R^{d^2\times d^2}\right)$, see Assumption \ref{a:ElastTen}, we obtain for any $\b v\in\b H^1_D(\Omega)$ the uniform estimate
	$$\left|\C\left(\varphi_{n_{k(l)}}(x)\right)\Epsilon\left(\b v\right)(x)\right|\leq C\left|\Epsilon(\b v)(x)\right|\quad\text{ for a.e. }x\in\Omega.$$
	Hence, Lebesgue's convergence theorem implies that $\left(\C\left(\varphi_{n_{k(l)}}\right)\Epsilon\left(\b v\right)\right)_{l\in\N}$ converges strongly in $L^2(\Omega)^{d\times d}$ to $\C\left(\varphi\right)\Epsilon\left(\b v\right)$. Since $\left(\Epsilon\left(\b u_{n_{k(l)}}\right)\right)_{l\in\N}$ converges additionally weakly in $L^2(\Omega)^{d\times d}$, we obtain that
	\begin{align*}
		&\lim_{l\to\infty}\left|\int_\Omega \C\left(\varphi_{n_{k(l)}}\right)\Epsilon\left(\b u_{n_{k(l)}}\right):\Epsilon\left(\b v\right)\dx-\int_\Omega \C\left(\varphi\right)\Epsilon\left(\b u\right):\Epsilon\left(\b v\right)\dx\right|=0.
	\end{align*}
	Similarly, we can deduce from the uniform boundedness of $\overline\Epsilon$ that
	\begin{align*}
		&\lim_{l\to\infty}\left|\int_\Omega \C\left(\varphi_{n_{k(l)}}\right)\overline\Epsilon\left(\varphi_{n_{k(l)}}\right):\Epsilon\left(\b v\right)\dx-\int_\Omega \C\left(\varphi\right)\overline\Epsilon\left(\varphi\right):\Epsilon\left(\b v\right)\dx\right|=0.
	\end{align*}
	This leads to
	$$\int_\Omega\C\left(\varphi\right)\left(\Epsilon\left(\b u\right)-\overline\Epsilon\left(\varphi\right)\right):\Epsilon\left(\b v\right)\dx=\int_\Omega\b f\cdot\b v\dx+\int_{\Gamma_g}\b g\cdot\b v\ds\quad\forall\b v\in\b H^1_D(\Omega)$$	
	which yields $\b u=\b S(\varphi)$. By applying the same arguments as above for any subsequence of $\left(\b S(\varphi_n)\right)_{n\in\N}$, we obtain that every subsequence of $\left(\b S(\varphi_n)\right)_{n\in\N}$ has a subsequence $\left(\b S\left(\varphi_{\hat n(k)}\right)\right)_{k\in\N}$ such that $\left(\b S\left(\varphi_{\hat n(k)}\right)\right)_{k\in\N}$ converges weakly in $\b H^1(\Omega)$ to $\b S(\varphi)=\b u$. This implies then the demicontinuity of $\b S$ as stated in the lemma.\\
	
	We are left with proving the continuity of $F_E$. For this purpose, we take again a sequence $\left(\varphi_k\right)_{k\in\N}\subset L^1(\Omega)$ such that $\left|\varphi_k\right|\leq1$ a.e. in $\Omega$ and $\lim_{k\to\infty}\left\|\varphi_k-\varphi\right\|_{L^1(\Omega)}=0$. We have already established, that this implies the weak convergence of $\left(\b S\left(\varphi_k\right)\right)_{k\in\N}$ to $\b S\left(\varphi\right)$ in $\b H^1(\Omega)$. Using the compact imbeddings $\b H^1(\Omega)\hookrightarrow\b L^2(\Omega)$ and $\b H^{\frac12}\left(\Gamma_g\right)\hookrightarrow\b L^2(\Gamma_g)$ we moreover find, that $\left(\b S\left(\varphi_k\right)\right)_{k\in\N}$ converges strongly in $\b L^2(\Omega)$ and $(\b S(\varphi_k)|_{\Gamma_g})_{k\in\N}$ converges strongly in $\b L^2(\Gamma_g)$. We can now use the continuity of the objective functional stated in Assumption~\ref{a:ElastObjectiveFctl}, see Remark~\ref{r:ContObjFctlElast}, to obtain
	$\lim_{k\to\infty}F_E\left(\varphi_k\right)=F_E\left(\varphi\right)$ and have shown the statement.
	\end{proof}

Using this lemma, we can show Theorem~\ref{t:ElastConvGammaConv} by applying known results concerning $\Gamma$-convergence of the Ginzburg-Landau energy.\\

{\em Proof of Theorem \ref{t:ElastConvGammaConv}:}
	By \cite{modica} we obtain, that the Ginzburg-Landau energy $E_\epsilon:L^1(\Omega)\to\overline\R$, which is given by
	$$E_\epsilon\left(\varphi\right):=\begin{cases}\int_\Omega\frac1\epsilon\psi\left(\varphi\right)+\frac\epsilon2\left|\nabla\varphi\right|^2\dx & \text{if }\varphi\in H^1(\Omega),\\+\infty &\text{otherwise},\end{cases}$$
	$\Gamma$-converges as $\epsilon\searrow0$ in $L^1(\Omega)$ to 
	$$E_0: L^1(\Omega)\ni\varphi\mapsto \begin{cases}c_0P_\Omega\left(\left\{\varphi=1\right\}\right) &\text{if }\varphi\in BV(\Omega,\{\pm1\}),\\+\infty & \text{else}.\end{cases}$$ We rewrite the reduced objective functional in the following form: $j_\epsilon=\gamma E_\epsilon+F_E+I_K$, where $I_K(\varphi):=0$ if $\varphi\in K$ and $I_K(\varphi)+\infty$ if $\varphi\in L^1(\Omega)\setminus K$ with $K:=\{\varphi\in L^1(\Omega)\mid\int_\Omega\varphi\dx\leq\beta\left|\Omega\right|\}.$
	Making use of Lemma \ref{l:ElastConvGammaConvFEcont}, we find that $F_E+I_K$ is a continuous function in $L^1(\Omega)$, and so $j_\epsilon$ is the Ginzburg-Landau energy $E_\epsilon$ plus some functional which is continuous in $L^1(\Omega)$. Consequently, by standard results for $\Gamma$-convergence, see for instance \cite{dalmaso}, we find that $\left(j_\epsilon\right)_{\epsilon>0}$ $\Gamma$-converges in $L^1(\Omega)$ to $j_0$, since $j_0(\varphi)=\gamma E_0\left(\varphi\right)+\left(F_E+I_K\right)\left(\varphi\right).$ This proves the statement.
\qquad$\square$\\

Now we want to prove the convergence of the equations of the first variation:\\

{\em Proof of Theorem~\ref{t:ConvOPtSys}:}
	The result of Corollary \ref{c:ConvMinElastConv} yields directly the existence of a subsequence of $\left(\varphi_\epsilon\right)_{\epsilon>0}$ converging in $L^1(\Omega)$ to a minimizer $\varphi_0$ of $j_0$ such that $\lim_{\epsilon\searrow0}j_\epsilon(\varphi_\epsilon)=j_0(\varphi_0)$. By Lemma \ref{l:ElastConvGammaConvFEcont}, this implies the weak convergence of $\left(\b u_\epsilon\right)_{\epsilon>0}$ to $\b u_0=\b S(\varphi_0)$ in $\b H^1(\Omega)$ as $\epsilon\searrow0$.\\
	Now we recall, that $\dot{\b u}_\epsilon\left[V\right]\in\b H^1_D(\Omega)$ is given as the solution of

\begin{equation}\label{e:ElastConvProofOptSysEquForDotUEpsin}\begin{split}&\int_\Omega\C(\varphi_\epsilon)\Epsilon(\dot{\b u}_\epsilon[V]):\Epsilon(\b v)\dx=\b R_\epsilon(\b v)\quad\forall\b v\in\b H^1_D(\Omega)\end{split}\end{equation}
where $\b R_\epsilon\in\left(\b H^1_D(\Omega)\right)'$ is given by
\begin{align*}&\b R_\epsilon(\b v):=\int_\Omega\C(\varphi_\epsilon)\frac12\left(\der\b u_\epsilon\der V(0)+\nabla V(0)\nabla\b u_\epsilon\right):\Epsilon(\b v)+\\
&+\C(\varphi_\epsilon)\left(\Epsilon(\b u_\epsilon)-\overline\Epsilon\left(\varphi_\epsilon\right)\right):\frac12(\nabla V(0)\nabla\b v+\der\b v\der V(0))-\\
&-\C(\varphi_\epsilon)\left(\Epsilon(\b u_\epsilon)-\overline\Epsilon\left(\varphi_\epsilon\right)\right):\Epsilon(\b v)\div V(0)\dx-\int_\Omega\b f\cdot\der \b vV(0)\dx-\int_{\Gamma_g}\b g\cdot\der\b vV(0)\ds.\end{align*}
Since $\left(\b u_\epsilon\right)_{\epsilon>0}$ is uniformly bounded in $\b H^1(\Omega)$, $\left\|\varphi_\epsilon\right\|_{L^\infty(\Omega)}\leq1$ and using the uniform estimate on the elasticity tensor and the eigenstrain given by Assumptions \ref{a:ElastTen} and \ref{a:Eigenstrain} we can deduce that $\sup_{\epsilon>0}\left\|\b R_\epsilon\right\|_{(\b H^1_D(\Omega))'}<\infty.$ And so we find by Korn's inequality from $\eqref{e:ElastConvProofOptSysEquForDotUEpsin}$ that $\sup_{\epsilon>0}\left\|\dot{\b u}_\epsilon\left[V\right]\right\|_{\b H^1(\Omega)}\leq C.$ This yields the existence of a subsequence, which will be denoted by the same, such that $\left(\dot{\b u}_\epsilon\left[V\right]\right)_{\epsilon>0}$ converges weakly in $\b H^1(\Omega)$ to $\b w\in\b H^1_D(\Omega)$. Following the arguments of the proof of Lemma \ref{l:ElastConvGammaConvFEcont} we see that the limit element $\b w$ of $\left(\dot{\b u}_\epsilon\left[V\right]\right)_{\epsilon>0}$ fulfills $\eqref{e:ElastPhaseDotUEpsilonEquationWeak}$. Hence, by definition of $\dot{\b u}_0\left[V\right]$, see Theorem \ref{t:ElastGeomVarSharp}, we get $\b w=\dot{\b u}_0\left[V\right]$. In particular, we can deduce by the imbedding theorems that both $\left(\b u_\epsilon\right)_{\epsilon>0}$ and $\left(\dot{\b u}_\epsilon\left[V\right]\right)_{\epsilon>0}$ converge strongly in $\b L^2(\Omega)$ and $\b L^2(\Gamma_g)$. And so we obtain by the continuous differentiability of the objective functional, see Remark~\ref{r:ElastObjFctFrechet}, that
 
 \begin{equation}\begin{split}
&\lim_{\epsilon\searrow0}\left[\int_\Omega\left[\der h_\Omega\left(x,\b u_\epsilon\right)\left(V(0),\dot{\b u}_\epsilon\left[V\right]\right)+h_\Omega\left(x,\b u_\epsilon\right)\div V(0)\right]\dx+\right.\\
	&\left.+\int_{\Gamma_g}\left[\der h_\Gamma\left(s,\b u_\epsilon\right)\left(V(0),\dot{\b u}_\epsilon\left[V\right]\right)+h_\Gamma\left(s,\b u_\epsilon\right)\left(\div V(0)-\b n\cdot\nabla V(0)\b n\right)\right]\ds\right]=\\
	&=\int_\Omega\left[\der h_\Omega\left(x,\b u_0\right)\left(V(0),\dot{\b u}_0\left[V\right]\right)+h_\Omega\left(x,\b u_0\right)\div V(0)\right]\dx+\\
	&+\int_{\Gamma_g}\left[\der h_\Gamma\left(s,\b u_0\right)\left(V(0),\dot{\b u}_0\left[V\right]\right)+h_\Gamma\left(s,\b u_0\right)\left(\div V(0)-\b n\cdot\nabla V(0)\b n\right)\right]\ds.
	\end{split}\end{equation}
 
 Analogously as in \cite{GarckeHechtStokes} we can apply the Reshetnyak continuity theorem to deduce 
 
 \begin{equation}\begin{split}
\lim_{\epsilon\searrow0}&\left[\int_\Omega\left(\frac{\gamma\epsilon}{2}\left|\nabla\varphi_\epsilon\right|^2+\frac\gamma\epsilon\psi\left(\varphi_\epsilon\right)\right)\div V(0)-\gamma\epsilon\nabla\varphi_\epsilon\cdot\nabla V(0)\nabla\varphi_\epsilon\dx\right]=\\
&=\gamma c_0\int_\Omega\left(\div V(0)-\nu\cdot\nabla V(0)\nu\right)\,\mathrm d\left|\der\chi_{E_0}\right|.
	\end{split}\end{equation}
	Plugging those results together we end up with $\lim_{\epsilon\searrow0}\partial_t|_{t=0} j_\epsilon\left(\varphi_\epsilon\circ T_t^{-1}\right)=\partial_t|_{t=0}j_0\left(\varphi_0\circ T_t^{-1}\right)$. As in \cite{GarckeHechtStokes} we can find some $V\in{\mathcal V}_{ad}$ such that $\int_\Omega\varphi_0\div V(0)\dx>0$ if we assume $|\{\varphi_0=1\}|>0$. Thus we have
	\begin{align*}\lim_{\epsilon\searrow0}-\lambda_\epsilon\int_\Omega\varphi_\epsilon\div V(0)\dx=\lim_{\epsilon\searrow0}\partial_t|_{t=0}j_\epsilon\left(\varphi_\epsilon\circ T_t^{-1}\right)=\partial_t|_{t=0}j_0\left(\varphi_0\circ T_t^{-1}\right)\end{align*}
	wherefrom we obtain that $\left(\lambda_\epsilon\right)_{\epsilon>0}$ converges to some $\lambda_0\geq0$. Besides, this directly yields that $\lambda_0\geq0$ fulfills $\eqref{e:ElastSharpOptCond1}$	and thus is a Lagrange multiplier associated to the integral constraint. This finally proves the statement. \qquad$\square$

\section{Conclusions}
We have shown that the proposed phase field approach leads to an optimal control problem for which existence of a solution can be shown. The problem can be reformulated in such a way that a reduced objective functional has to be minimized. The latter $\Gamma$-converges in $L^1(\Omega)$ as the thickness of the interface tends to zero to a functional describing a sharp interface formulation of the problem. We have shown that certain first order optimality conditions for the phase field problem can be deduced by geometric variations. As the minimizers converge, also the obtained optimality conditions converge to a system, which is a necessary optimality condition for the sharp interface problem. Besides, this optimality system for the sharp interface problem can be derived in the general setting of functions of bounded variations.\\
Assuming additional regularity assumptions on the minimizing set and the data, it can be shown that the obtained conditions are equivalent to results that were already obtained in literature by classical shape calculus and also by formal asymptotics from the phase field model. Thus we have delivered a rigorous proof for the convergence results that were already predicted by formal asymptotics in \cite{relatingphasefield}. Moreover we use a general objective functional. However, in \cite{relatingphasefield} the state constraints can be $\epsilon$-dependent. To be precise, an ersatz material approach is used, where the stiffness of the ersatz material scales like $\epsilon^2$, and thus vanishes as $\epsilon\searrow0$. This is not done in our work, but possible generalizations for reasonable objective functionals in the spirit of \cite{HechtStokesEnergy, GarckeHechtStokes} may be possible. This means that convergence of minimizers could possibly be shown, but we expect that again certain growth conditions on the convergence of the minimizers play a role, where this rate has to be consistent with the $\epsilon$-scaling of the ersatz material.\\
We presented numerical simulations which were obtained with the help of a projected gradient type method which showed that the proposed phase field approach works well in two and three spatial dimensions.

\bibliographystyle{siam}
\bibliography{literature}

\begin{thebibliography}{10}

\bibitem{allairemulti}
{\sc G.~Allaire, C.~Dapogny, G.~Delgado, and G.~Michailidis}, {\em {Multi-phase
  structural optimization via a level set method}}, COCV, 20 (2014),
  pp.~576--611.

\bibitem{allaire_jouve}
{\sc G.~Allaire and F.~Jouve}, {\em {A level-set method for vibration and
  multiple loads structural optimization}}, Comput. Methods Appl. Mech. Engrg.,
  194 (2005), pp.~3269--3290.

\bibitem{Allaire2004363}
{\sc G.~Allaire, F.~Jouve, and A.-M. Toader}, {\em Structural optimization
  using sensitivity analysis and a level-set method}, Journal of Computational
  Physics, 194 (2004), pp.~363 -- 393.

\bibitem{allaire2010damage}
{\sc G.~Allaire, F.~Jouve, and N.~Van~Goethem}, {\em Damage evolution in
  brittle materials by shape and topological sensitivity analysis}, J. Comput.
  Phys., 230 (2011), pp.~5010--5044.

\bibitem{ambrosioButtazzo}
{\sc L.~Ambrosio and G.~Buttazzo}, {\em {An optimal design problem with
  perimeter penalization}}, Calc. Var. Partial Differential Equations, 1
  (1993), pp.~55--69.

\bibitem{ambrosio}
{\sc L.~Ambrosio, N.~Fusco, and D.~Pallara}, {\em {Functions of Bounded
  Variation and Free Discontinuity Problems}}, Oxford: Clarendon Press, 2000.

\bibitem{BNS04}
{\sc J.~Barrett, R.~N\"urnberg, and V.~Styles}, {\em {Finite Element
  Approximation of a Phase Field Model for Void Electromigration}}, SIAM J.
  Numer. Anal., 42 (2004), pp.~738--772.

\bibitem{bendsoe2003topology}
{\sc M.~Bends{\o}e}, {\em {Topology optimization: theory, methods and
  applications}}, Springer, 2003.

\bibitem{haberjog}
{\sc M.~Bends{\o}e, R.~Haber, and C.~Jog}, {\em {A new approach to
  variable-topology shape design using a constraint on perimeter}}, Struct.
  Multidiscip. Optim., 11 (1996), pp.~1--12.

\bibitem{ItoKunischPDAS}
{\sc M.~Bergounioux, K.~Ito, and K.~Kunisch}, {\em Primal-dual strategy for
  constrained optimal control problems}, SIAM Journal on Control and
  Optimization, 37 (1999), pp.~1176--1194.

\bibitem{relatingphasefield}
{\sc L.~Blank, H.~Farshbaf-Shaker, H.~Garcke, and V.~Styles}, {\em Relating
  phase field and sharp interface approaches to structural topology
  optimization}, ESAIM: COCV, 20 (2014), pp.~1024--1058.

\bibitem{Buchkapitel}
{\sc L.~Blank, M.~Farshbaf-Shaker, H.~Garcke, C.~Rupprecht, and V.~Styles},
  {\em {Multi-material phase field approach to structural topology
  optimization}}, in Trends in PDE Constrained Optimization, G.~Leugering,
  P.~Benner, S.~Engell, A.~Griewank, H.~Harbrecht, M.~Hinze, R.~Rannacher, and
  S.~Ulbrich, eds., ISNM, Birkh\"auser, to appear.

\bibitem{bgsssv2010}
{\sc L.~Blank, H.~Garcke, L.~Sarbu, T.~Srisupattarawanit, V.~Styles, and
  A.~Voigt}, {\em Phase-field approaches to structural topology optimization},
  in Constrained Optimization and Optimal Control for Partial Differential
  Equations, G.~Leugering, S.~Engell, M.~Hinze, R.~Rannacher, V.~Schulz,
  V.~Ulbrich, and S.~Ulbrich, eds., vol.~160 of ISNM, Birkh\"auser, 2012,
  pp.~245--256.

\bibitem{RupprechtBlank}
{\sc L.~Blank and C.~Rupprecht}, {\em {An extension of the projected gradient
  method to a Banach space setting with application in structural topology
  optimization}}.
\newblock arXiv:1503.03783v2.

\bibitem{blowey_elliot}
{\sc J.~F. Blowey and C.~M. Elliott}, {\em {The Cahn-Hilliard gradient theory
  for phase separation with non-smooth free energy Part I: Mathematical
  analysis}}, European J. Appl. Math., 2 (1991), pp.~233--280.

\bibitem{bourdin_chambolle}
{\sc B.~Bourdin and A.~Chambolle}, {\em {Design-dependent loads in topology
  optimization}}, ESAIM Control Optim. Calc. Var., 9 (2003), pp.~19--48.

\bibitem{braess}
{\sc D.~Braess}, {\em {Finite Elemente}}, Springer, 1997.

\bibitem{ciarlet}
{\sc P.~Ciarlet}, {\em Three-Dimensional Elasticity}, vol.~1 of Studies in
  mathematics and its applications, Elsevier Science, 1988.

\bibitem{dalmaso}
{\sc G.~Dal~Maso}, {\em {An Introduction to $\Gamma$-convergence}}, Progress in
  Nonlinear Differential Equations and Their Applications, Birkh{\"a}user,
  1993.

\bibitem{delfour}
{\sc M.~Delfour and J.~Zol{\'e}sio}, {\em {Shapes and Geometries: Analysis,
  Differential Calculus, and Optimization}}, Adv. Des. Control, SIAM, 2001.

\bibitem{delfour2}
\leavevmode\vrule height 2pt depth -1.6pt width 23pt, {\em {Shapes and
  Geometries: Metrics, Analysis, Differential Calculus and Optimization}}, Adv.
  Des. Control, SIAM, 2011.

\bibitem{deflourpaper}
{\sc M.~Delfour and J.~Zolésio}, {\em Shape derivatives for nonsmooth domains},
  in Optimal Control of Partial Differential Equations, K.-H. Hoffmann and
  W.~Krabs, eds., vol.~149 of Lecture Notes in Control and Inform. Sci.,
  Springer, 1991, pp.~38--55.

\bibitem{dorand1922influence}
{\sc C.~Dorand}, {\em {Influence of Elliptical Distribution of Lift on Strength
  of Airplane Wings}}, Technical memorandum, National Advisory Committee for
  Aeronautics, 1922.

\bibitem{eckgarcke}
{\sc C.~Eck, H.~Garcke, and P.~Knabner}, {\em {Mathematische Modellierung}},
  Springer, 2008.

\bibitem{evans_gariepy}
{\sc L.~Evans and R.~Gariepy}, {\em {Measure Theory and Fine Properties of
  Functions}}, Mathematical Chemistry Series, CRC PressINC, 1992.

\bibitem{garcke_paper}
{\sc H.~Garcke}, {\em {The $\Gamma$-limit of the Ginzburg-Landau energy in an
  elastic medium}}, AMSA, 18 (2008), pp.~345--379.

\bibitem{HechtStokesEnergy}
{\sc H.~Garcke and C.~Hecht}, {\em {A phase field approach for shape and
  topology optimization in Stokes flow}}, Preprint-Nr.: 09/2014, Universität
  Regensburg, Mathematik,  (2014).

\bibitem{GarckeHechtStokes}
\leavevmode\vrule height 2pt depth -1.6pt width 23pt, {\em {Shape and topology
  optimization in Stokes flow with a phase field approach}}, Preprint-Nr.:
  10/2014, Universität Regensburg, Mathematik,  (2014).

\bibitem{giusti}
{\sc E.~Giusti}, {\em {Minimal surfaces and functions of bounded variation}},
  Notes on pure mathematics, Dept. of Pure Mathematics, 1977.

\bibitem{hecht}
{\sc C.~Hecht}, {\em {Shape and topology optimization in fluids using a phase
  field approach and an application in structural optimization}},
  {Dissertation}, University of Regensburg, 2014.

\bibitem{sturm}
{\sc M.~Hintermüller, D.~Hömberg, and K.~Sturm}, {\em {Shape optimization for a
  sharp interface model of distortion compensation}}.
\newblock WIAS Preprint No. 1792, 2013.

\bibitem{michell}
{\sc A.~G.~M. Michell}, {\em The limits of economy of material in
  frame-structures}, Phil. Mag., 8 (1904), pp.~589--597.

\bibitem{modica}
{\sc L.~Modica}, {\em {The gradient theory of phase transitions and the minimal
  interface criterion}}, Arch. Ration. Mech. Anal., 98 (1987), pp.~123--142.

\bibitem{modica_mortola}
{\sc L.~Modica and S.~Mortola}, {\em {Un esempio di $\Gamma$-convergenza}},
  Boll. Un. Mat. Ital. B (5), 14 (1977), pp.~285--299.

\bibitem{sale2013structural}
{\sc D.~Sale, A.~Aliseda, M.~Motley, and Y.~Li}, {\em {Structural Optimization
  of Composite Blades for Wind and Hydrokinetic Turbines}}, Proceedings of the
  1st Marine Energy Technology Symposium METS 2013,  (2013).

\bibitem{showalter}
{\sc R.~Showalter}, {\em {Monotone Operators in Banach Spaces and Nonlinear
  Partial Differential Equations}}, Mathematical surveys and monographs, v. 49,
  American Mathematical Society, 1997.

\bibitem{thomson}
{\sc J.~Thomsen}, {\em Topology optimization of structures composed of one or
  two materials}, Structural optimization, 5 (1992), pp.~108--115.

\end{thebibliography}

\end{document}